\documentclass[12pt]{amsart}
\usepackage{a4wide,enumerate,xcolor,graphicx}
\usepackage{amsmath,comment}
\usepackage{scrextend} 
\usepackage[notcite,notref]{showkeys}
\usepackage{array,multirow,makecell}
\usepackage[normalem]{ulem}
\usepackage{soul}
\allowdisplaybreaks

\let\pa\partial
\let\na\nabla
\let\eps\varepsilon
\newcommand{\N}{{\mathbb N}}
\newcommand{\R}{{\mathbb R}}

\newcommand{\diver}{\operatorname{div}}
\newcommand{\T}{{\mathcal T}}
\newcommand{\E}{{\mathcal E}}
\newcommand{\m}{\operatorname{m}}
\newcommand{\Eint}{{\mathcal E}_{{\rm int},K}}
\newcommand{\dist}{{\operatorname{d}}}

\newcolumntype{C}[1]{>{\centering\arraybackslash }b{#1}}

\newtheorem{theorem}{Theorem}
\newtheorem{lemma}[theorem]{Lemma}

\newtheorem{proposition}[theorem]{Proposition}
\newtheorem{remark}[theorem]{Remark}
\newtheorem{corollary}[theorem]{Corollary}

\begin{document}

\title[Finite-volume scheme for the SKT system]{A convergent structure-preserving
finite-volume scheme for the Shigesada-Kawasaki-Teramoto population system}

\author[A. J\"ungel]{Ansgar J\"ungel}
\address{Institute for Analysis and Scientific Computing, Vienna University of
	Technology, Wiedner Hauptstra\ss e 8--10, 1040 Wien, Austria}
\email{juengel@tuwien.ac.at}

\author[A. Zurek]{Antoine Zurek}
\address{Institute for Analysis and Scientific Computing, Vienna University of
	Technology, Wiedner Hauptstra\ss e 8--10, 1040 Wien, Austria}
\email{antoine.zurek@tuwien.ac.at}

\date{\today}

\thanks{The authors have been partially supported by the Austrian Science Fund (FWF), 
grants P30000, P33010, F65, and W1245.}

\begin{abstract}
An implicit Euler finite-volume scheme for an $n$-species population 
cross-diffusion system of
Shigesada--Kawasaki--Teramoto-type in a bounded domain with no-flux boundary
conditions is proposed and analyzed. The scheme
preserves the formal gradient-flow or entropy structure and preserves the
nonnegativity of the population densities. The key idea is to consider a suitable 
mean of the mobilities in such a way that a discrete chain rule is fulfilled and 
a discrete analog of the entropy inequality holds. 
The existence of finite-volume solutions, the convergence of the scheme,
and the large-time asymptotics to the constant steady state are proven.
Furthermore, numerical experiments in one and two space dimensiona
for two and three species are presented.
The results are valid for a more general class of cross-diffusion systems
satisfying some structural conditions.
\end{abstract}



\maketitle


\section{Introduction}

The population model of Shigesada, Kawasaki, and Teramoto (SKT) 
describes the segregation of two competing species \cite{SKT79}. 
It consists of quasilinear
parabolic equations for the population densities with a generally nonsymmetric 
and not positive semidefinite diffusion matrix. To overcome the lack of positive 
definiteness,
it was suggested in \cite{ChJu04,GGJ03} to use so-called entropy variables
that yield a transformed diffusion system with a positive semidefinite diffusion matrix.
In particular, the SKT cross-diffusion system of \cite{SKT79}
has a formal gradient-flow or entropy structure. 
This approach can be generalized to an arbitrary number of species \cite{CDJ18}.
It is important to design a general easy-to-implement numerical scheme that preserves 
this structure and that can be proven to be convergent.
Previous works like \cite{ABR11,BaBl04,SCW19} propose numerical approximations
that satisfy some of these properties but not all of them. In this paper,
we suggest a finite-volume scheme for $n$-species SKT-type population systems,
preserving the entropy structure and the nonnegativity of densities and conserving
the mass (in the absence of source terms). In fact, our results are even valid for
a more general class of cross-diffusion systems satisfying some structural
conditions.

More precisely, we consider the cross-diffusion system
\begin{equation}\label{1.eq}
  \pa_t u_i - \diver\bigg(\sum_{j=1}^n A_{ij}(u)\na u_j\bigg) = f_i(u)
  \quad\mbox{in }\Omega,\ t>0,\ i=1,\ldots,n,
\end{equation}
where $\Omega\subset\R^d$ ($d\geq 1$) is a bounded domain, $u=(u_1,\ldots,u_n)$ 
is the vector of population densities, with the diffusion coefficients
\begin{equation}\label{1.A}
  A_{ij}(u) = \delta_{ij}\bigg(a_{i0} + \sum_{k=1}^n a_{ik}u_k\bigg)
	+ a_{ij}u_i, \quad i,j=1,\ldots,n,
\end{equation}
and the Lotka--Volterra source terms,
\begin{equation}\label{1.f}
  f_i(u) = u_i\bigg(b_{i0} - \sum_{j=1}^n b_{ij}u_j\bigg), \quad i=1,\ldots,n,
\end{equation}
where we assume that $a_{ii}>0$, $b_{ii}>0$ for $i=1,\ldots,n$ and
$a_{ij}\ge 0$ and $b_{ij}\ge 0$ for $i\neq j$.
We prescribe no-flux boundary and initial conditions:
\begin{equation}\label{1.bic}
  \sum_{j=1}^n A_{ij}(u)\na u_j\cdot\nu = 0\quad\mbox{on }\pa\Omega,\ t>0, \quad
	u_i(0)=u_i^0\quad\mbox{in }\Omega, \quad i=1,\ldots,n,
\end{equation}
where $\nu$ denotes the exterior unit normal vector to $\pa\Omega$.
When $n=2$, we recover the SKT system of \cite{SKT79} without environmental
potentials. Our analysis also works when we include the corresponding drift terms 
(see Section \ref{ssec.niche}).

Let $h:[0,\infty)^n\to [0,\infty)$ be a convex function and set 
$H[u]=\int_\Omega h(u)dx$.
The entropy inequality is derived, for suitable source terms, by choosing $h'(u)$ 
formally as a test function in the weak formulation of \eqref{1.eq}, leading to
\begin{equation}\label{1.ei}
  \frac{dH}{dt} 
	+ \int_\Omega \na u:h''(u)A(u)\na u dx \le C(T), \quad 0<t<T,
\end{equation}
where $h''(u)$ is the Hessian of $h$, ``:'' is the Frobenius matrix product,
and $C(T)>0$ comes from the source terms.
We call $H$ an entropy and $h$ an entropy density if $h''(u)A(u)$ is 
positive (semi-) definite. This typically provides gradient estimates and moreover,
if $C(T)=0$, then $H$ is a Lyapunov functional along the solutions to \eqref{1.eq}.

In the case of the $n$-species SKT model, the entropy density is given by
\begin{equation}\label{1.h}
  h(u) = \sum_{i=1}^n\pi_i \big(u_i(\log u_i-1)+1\big), \quad u\in[0,\infty)^n,
\end{equation}
where the numbers $\pi_i>0$ are assumed to satisfy $\pi_i a_{ij}=\pi_j a_{ji}$ 
for $i\neq j$.
This can be recognized as the detailed-balance condition for the time-continuous
Markov chain associated to $(a_{ij})$, and the vector $(\pi_i)$ is the
corresponding invariant measure \cite{CDJ18}. 
It turns out that $\na u:h''(u)A(u)\na u$ is bounded from below by 
$\sum_{i=1}^n a_{ii}|\na u_i|^2$, which yields $H^1(\Omega)$ estimates.
Moreover, it can be shown that the solutions $u_i$ are nonnegative and
the mass $\int_\Omega u_i(t)dx$ is constant in time if $f_i=0$.
Our aim is to preserve this structure on the discrete level.

In the literature, there are already various numerical schemes for the
SKT model. Up to our knowledge, the first numerical simulations,
based on a finite-difference scheme in one space dimension, were
performed in \cite{GGJ01}. A convergence result for an implicit Euler approximation,
which preserves the nonnegativity of the densities, was proved in \cite{GGJ03},
but the space variable was not discretized. Based on the entropy structure found
in \cite{ChJu04,GGJ03}, a convergent entropy-dissipative finite-element approximation
was proposed in \cite{BaBl04}. The entropy structure is preserved by defining
an approximation of a certain mean function.
For this, the authors of \cite{BaBl04} need an approximated entropy and an
approximated diffusion matrix, which complicates the numerical scheme. Moreover,
their scheme does not preserve the nonnegativity of the densities.
A convergent finite-volume scheme that preserves the nonnegativity was suggested in
\cite{ABR11}, but the analysis is valid only for positive definite diffusion
matrices $(A_{ij}(u))$, which requires strong conditions on $a_{ij}$.
Another idea was developed in \cite{Mur17}, by considering a linear finite-volume
scheme and proving unconditional stability and convergence, but without 
structure-preserving properties. 
A discontinuous Galerkin scheme was used in \cite{SCW19}, which
preserves the formal gradient-flow structure and nonnegativity of the densities,
but no convergence analysis was performed. Finally, operator-splitting 
techniques were also applied to the SKT model \cite{BePa19,GLS09}.

Compared to the literature, our finite-volume scheme 
(i) preserves the entropy structure of the $n$-species model under the
detailed-balance condition, (ii) preserves the nonnegativity of $u_i\ge 0$,
and (iii) conserves the mass when the source terms vanish. 
We design and analyze in fact a finite-volume scheme for a general
cross-diffusion model of the form \eqref{1.eq} and \eqref{1.bic}, satisfying
some structural conditions specified in Section \ref{sec.main}. 
For this scheme, we prove the
existence of discrete finite-volume solutions and show that a subsequence
converges to the solutions to \eqref{1.eq} and \eqref{1.bic}. 
In Section \ref{sec.exam}, we apply the results obtained in the 
general framework to the SKT model \eqref{1.eq}-\eqref{1.bic}.

The derivation of the entropy inequality \eqref{1.ei} is based on the
chain rule $h''(u)\na u=\na h'(u)$. The difficulty is to formulate this
identity on the discrete level. Let $\Omega$ be the union of
cells $K$ and let $\sigma=K|L$ be the edge between two neighboring
cells $K$ and $L$. The finite-volume density $u_i$ is constant on
each cell, and we write $u_{i,K}$ for its value and set $u_K=(u_{1,K},\ldots,u_{n,K})$.
A discrete analog of the chain rule is the vector-valued identity
$$
  h''(\widetilde u_\sigma)(u_L-u_K) = h'(u_L)-h'(u_K),
$$
where $\widetilde u_\sigma$ is a mean vector. 
This approach resembles the discrete-gradient method \cite[Section V.5]{HLW06}.
However, the mean-value theorem
for vector-valued functions can be formulated only as
$$
  \bigg(\int_0^1 h''(su_L + (1-s)u_K)ds\bigg)(u_L-u_K) = h'(u_L)-h'(u_K),
$$
and in general, a mean vector $\widetilde u_\sigma$ cannot be found.
Therefore, we assume that the entropy density is the sum of entropy densities
for each species, $h(u)=\sum_{i=1}^n h_i(u_i)$. Then the Hessian of $h$ is diagonal,
and the standard mean-value theorem can be applied componentnwise. 
Fortunately, the entropy
\eqref{1.h} of the SKT model satisfies this condition. In this case, the mean
vector is computed by
$$
  \widetilde u_{i,\sigma} = \frac{u_{i,L}-u_{i,K}}{\log u_{i,L}-\log u_{i,K}}\quad
	\mbox{for }\sigma=K|L\mbox{ if }u_{i,K}\neq u_{i,L}.
$$
This corresponds to the logarithmic mean, used in, e.g., \cite{Mie13}. 
General mean functions are defined in, e.g., \cite{CCHK20,FiHe17,GrRu00}.
In order to achieve a discrete analog of the entropy inequality \eqref{1.ei}, 
the diffusion matrix has to be evaluated
at the mean vector $\widetilde u_\sigma$, 
i.e., the fluxes of the finite-volume 
scheme along the edge $\sigma=K|L$ have to be discretized according to 
$$
  \mathcal{F}_{i,K,\sigma} = -\sum_{j=1}^n\tau_\sigma A_{ij}(\widetilde u_\sigma)
	(u_{j,L}-u_{j,K}), \quad i=1,\ldots,n,
$$
where $\tau_\sigma$ is the transmissibility constant defined in 
\eqref{2.trans} below.

The paper is organized as follows. The numerical scheme and our main results
(existence of discrete solutions, convergence of the scheme, and 
large-time behavior) are introduced
in Section \ref{sec.main}. Examples that satisfy our general assumptions,
including the SKT model, are presented in Section \ref{sec.exam}.
In Section \ref{sec.ex}, we prove the existence of discrete solutions.
Uniform estimate are derived in Section \ref{sec.est}, and Section \ref{sec.conv}
is devoted to the proof of the convergence of the scheme. 
The large-time asymptotics is shown in Section \ref{sec.time}.
Finally, we present in Section 
\ref{sec.numer} some numerical examples for the two- and three-species SKT system.


\section{Numerical scheme and main results}\label{sec.main}

\subsection{Notation and definitions}

We present the discretization of the domain $\Omega_T=\Omega\times(0,T)$.
We consider only two-dimensional domains $\Omega$, but the generalization
to higher space dimensions is straightforward. Let $\Omega\subset\R^2$ be a bounded, 
polygonal domain. An admissible mesh of $\Omega$ is given by (i) a family $\T$
of open polygonal control volumes (or cells), (ii) a family $\E$ of edges, and
(iii) a family ${\mathcal P}$ of points $(x_K)_{K\in\T}$ associated to the
control volumes and satisfying Definition 9.1 in \cite{EGH00}. This definition
implies that the straight line $\overline{x_Kx_L}$ between two centers of
neighboring cells is orthogonal to the edge $\sigma=K|L$ between two cells.
For instance, Vorono\"{\i} meshes satisfy this condition \cite[Example 9.2]{EGH00}.
The size of the mesh is denoted by $\Delta x = \max_{K\in\T}\operatorname{diam}(K)$.
The family of edges $\E$ is assumed to consist of interior edges $\E_{\rm int}$
satisfying $\sigma\in\Omega$ and boundary edges $\sigma\in\E_{\rm ext}$ satisfying
$\sigma\subset\pa\Omega$. For given $K\in\T$, $\E_K$ is the set of edges of $K$,
and it splits into $\E_K=\Eint\cup\E_{{\rm ext},K}$. 
For any $\sigma\in\E$, there exists at least one cell $K\in\T$ such that 
$\sigma\in\E_K$.

We need the following definitions. For $\sigma\in\E$, we introduce the distance
$$
  \dist_\sigma = \begin{cases}
	\dist(x_K,x_L) &\quad\mbox{if }\sigma=K|L\in\E_{{\rm int},K}, \\
	\dist(x_K,\sigma) &\quad\mbox{if }\sigma\in\E_{{\rm ext},K},
	\end{cases}
$$
where d is the Euclidean distance in $\R^2$, and the transmissibility coefficient
\begin{equation}\label{2.trans}
  \tau_\sigma = \frac{\m(\sigma)}{\dist_\sigma},
\end{equation}
where $\m(\sigma)$ denotes the Lebesgue measure of $\sigma$.
The mesh is assumed to satisfy the following regularity assumption: There exists
$\zeta>0$ such that for all $K\in\T$ and $\sigma\in\E_K$,
\begin{equation}\label{2.dd}
  \dist(x_K,\sigma)\ge \zeta \dist_\sigma.
\end{equation}

Let $T>0$, let $N_T\in\N$ be the number of time steps, and introduce the
step size $\Delta t=T/N_T$ as well as the time steps $t_k=k\Delta t$ 
for $k=0,\ldots,N_T$. We denote by ${\mathcal D}$
the admissible space-time discretization of $\Omega_T$ composed of an admissible
mesh ${\mathcal T}$ and the values $(\Delta t,N_T)$. 

We also introduce suitable function spaces for the numerical scheme. The space of
piecewise constant functions is defined by 
$$
  \mathcal{H}_\T = \bigg\{v: \Omega\to\R:\exists (v_K)_{K\in\T}\subset\R,\
	v(x)=\sum_{K\in\T}v_K\mathbf{1}_K(x)\bigg\},
$$
where $\mathbf{1}_K$ is the characteristic function on $K$.
In order to define a norm on this space, we first introduce the notation
$$
  v_{K,\sigma} = \begin{cases}
	v_L &\quad\mbox{if }\sigma=K|L\in\Eint, \\
	v_K &\quad\mbox{if }\sigma\in\E_{{\rm ext},K},
	\end{cases} 
$$
for $K\in\T$, $\sigma\in\E_K$ and the discrete operators
$$
  \textrm{D}_{K,\sigma} v := v_{K,\sigma}-v_K, \quad 
	\textrm{D}_\sigma v := |\mathrm{D}_{K,\sigma} v|.
$$

Let $q\in[1,\infty)$ and $v\in\mathcal{H}_\T$. The discrete $W^{1,q}$ seminorm 
and discrete $W^{1,q}$ norm on $\mathcal{H}_\T$ are given by 
$$
  |v|_{1,q,\T}^q = \sum_{\sigma\in\E}\m(\sigma)\,\dist_{\sigma} 
	\bigg|\frac{\textrm{D}_{\sigma} v}{\dist_{\sigma}}\bigg|^q, \quad 
	\|v\|_{1,q,\T}^q = |v|^q_{1,q,\T} + \|v\|^q_{0,q,\T},
$$
respectively, and $\|v\|_{0,q,\T}$ denotes the $L^q$ norm 
i.e.\ $\|v\|_{0,q,\T} = (\sum_{K\in\T}\m(K)|v_K|^q)^{1/q}$. 
For given $q>1$, we associate to these norms a dual norm with respect to the
$L^2$ inner product,
$$
  \|v\|_{-1,q',\T} = \sup\bigg\{\int_\Omega vw dx: w\in\mathcal{H}_\T,\
	\|w\|_{1,q,\T}=1\bigg\},
$$
where $1/q+1/q'=1$. Then
$$
  \bigg|\int_\Omega vw dx\bigg| \le \|v\|_{-1,q',\T}\|w\|_{1,q,\T}
	\quad\mbox{for }v,w\in\mathcal{H}_\T.
$$
Finally, we introduce the space $\mathcal{H}_{\T,\Delta t}$ of piecewise constant
in time functions with values in $\mathcal{H}_\T$,
$$
  \mathcal{H}_{\T,\Delta t} = \bigg\{ v:\overline\Omega\times[0,T]\to\R:
	\exists(v^k)_{k=1,\ldots,N_T}\subset\mathcal{H}_\T,\
	v(x,t) = \sum_{k=1}^{N_T} v^k(x)\mathbf{1}_{(t_{k-1},t_k]}(t)\bigg\},
$$
equipped, for $1 \leq p,q < \infty$, with the discrete $L^p(0,T;W^{1,q}(\Omega))$ norm
$$
  \bigg(\sum_{k=1}^{N_T}\Delta t \|v^k\|_{1,q,\T}^{p}\bigg)^{1/p}.
$$


\subsection{Numerical scheme}

We define now the finite-volume scheme for the cross-diffusion model \eqref{1.eq} 
and \eqref{1.bic}, where we consider a general diffusion matrix $A(u)$ and an 
entropy density $h$ given by $h(u)=\sum_{i=1}^n h_i(u_i)$. 
We first approximate the initial functions by
\begin{equation}\label{2.init}
  u_{i,K}^0 = \frac{1}{\m(K)}\int_K u_i^0(x)dx \quad\mbox{for }K\in\T,\ 
	i=1,\ldots,n.
\end{equation}
Let $u^{k-1}=(u_1^{k-1},\ldots,u_n^{k-1})$ be given. Then the values $u_{i,K}^k$
are determined by the implicit Euler finite-volume scheme
\begin{equation}\label{2.fvm}
  \m(K)\frac{u_{i,K}^k-u_{i,K}^{k-1}}{\Delta t} 
	+ \sum_{\sigma\in\E_K}\mathcal{F}_{i,K,\sigma}^k = \m(K) f_i(u^k_K),
\end{equation}
where the fluxes $\mathcal{F}_{i,K,\sigma}^k$ are given by
\begin{equation}\label{2.flux}
  \mathcal{F}_{i,K,\sigma}^k = - \sum_{j=1}^n \tau_\sigma A_{ij}(u^k_\sigma) 
	\textrm{D}_{K,\sigma} u^k_j \quad\mbox{for }K\in \T,\ \sigma \in \E_K,
\end{equation}
and $\tau_\sigma$ is defined by \eqref{2.trans}. By the definition of the
discrete gradient $\textrm{D}_{K,\sigma}$, the discrete fluxes vanish on
the boundary edges, guaranteeing the no-flux boundary conditions. In \eqref{2.flux},
we have introduced the mean value
\begin{equation}\label{2.usigma}
  u^k_{i,\sigma} = \begin{cases}
  \widetilde u_{i,\sigma}^k
	\quad & \mbox{if } u^k_{i,K} > 0, \ u^k_{i,K,\sigma} > 0, \mbox{ and } 
	u^k_{i,K} \neq u^k_{i,K,\sigma}, \\
  u^k_{i,K} & \mbox{if } u^k_{i,K} = u^k_{i,K,\sigma} > 0, \\
  0 & \mbox{else},
  \end{cases}   
\end{equation}
where $\widetilde u_{i,K}^k\in(0,\infty)$ is the unique solution to
\begin{equation}\label{2.chain}
  h_i''(\widetilde u_{i,\sigma}^k)\textrm{D}_{K,\sigma}u_i^k
	= \textrm{D}_{K,\sigma} h_i'(u_i^k)\quad\mbox{for }K\in\T,\ \sigma\in\E_K.
\end{equation}
Since $h'_i$ is assumed to be strictly concave 
(see Hypothesis (H4) below), the definition $u_{i,\sigma}^k=0$ if $u_{i,K}^k=0$
or $u_{i,K,\sigma}^k=0$ is consistent with \eqref{2.chain}, and
the existence of a unique value $\widetilde u_{i,\sigma}^k$ follows 
from the mean-value theorem. The strict concavity of $h_i'$ (which implies
that $h_i''$ is strixtly decreasing) and
$$
  h_i''(\min\{u_{i,K}^k,u_{i,K,\sigma}^k\})
  \le h_i''(u_{i,\sigma}^k) \le  h_i''(\max\{u_{i,K}^k,u_{i,K,\sigma}^k\})
$$
lead to the bounds
\begin{equation}\label{2.est.usigma}
  \min\{u_{i,K}^k,u_{i,K,\sigma}^k\}\le
	\widetilde u_{i,\sigma}^k\le\max\{u_{i,K}^k,u_{i,K,\sigma}^k\}.
\end{equation}


\subsection{Main results}\label{ssec.main}

Our hypotheses are as follows.

\begin{labeling}{(A44)}
\item[\textbf{(H1)}] Domain: $\Omega\subset\R^2$ is a bounded polygonal domain.

\item[\textbf{(H2)}] Discretization: $\mathcal{D}$ is an admissible discretization of 
$\Omega_T$ satisfying \eqref{2.dd}.

\item[\textbf{(H3)}] Initial data: $u^0=(u^0_1,\ldots,u_n^0) \in 
L^1(\Omega;[0,\infty)^n)$ with $\int_\Omega h(u^0) dx < \infty$.

\item[\textbf{(H4)}] Entropy density: $h(u)=\sum_{i=1}^n h_i(u_i)$, where
$h_i\in C^2((0,\infty);(0,\infty))\cap C^0([0,\infty);$ $[0;\infty))$ is 
convex, $h_i':(0,\infty)\to\R$ is invertible and strictly concave,
and there exists $c_h>0$ such that $h_i(s)\ge c_h (s-1)$ for all $s\ge 0$,
$i=1,\ldots,n$.

\item[\textbf{(H5)}] Diffusion matrix: $A\in C^{0,1}([0,\infty)^n;\R^{n\times n})$ 
and there exists $c_A>0$ such that for all $z\in\R^n$ and $u\in(0,\infty)^n$,
$$
  z^\top h''(u)A(u)z \ge c_A|z|^2.
$$

\item[\textbf{(H6)}] Source terms: $f_i\in C^0([0,\infty))$, and there exist 
two constants $C_f>0$ and $C'_f>0$ such that for all $u\in[0,\infty)^n$,
$$
  \sum_{i=1}^n f_i(u)h'_i(u_i)\le C_f(1+h(u)), \quad
	\sum_{i=1}^n|f_i(u)|\le C'_f\bigg(1+\sum_{i=1}^n|u_i|^2\bigg).
$$
\end{labeling}

Let us discuss these hypotheses. The convexity of $h_i$ and the invertibility
of $h_i'$ in Hypothesis (H4) are natural conditions for the entropy method,
see \cite{Jue15,Jue16}. The strict convexity or concavity of $h'_i$ is required 
to define properly the mean value $\widetilde u^k_{i,\sigma}$ in \eqref{2.usigma}.
The lower bound for $h_i$ allows us to conclude $L^1(\Omega)$ estimates.
We assume in Hypothesis (H5) that the matrix $h''(u)A(u)$ is positive definite.
This condition can be relaxed, at least for the existence proof, to 
the ``degenerate'' positive definiteness assumption
$z^\top h''(u)A(u)z\ge c_A\sum_{i=1}^n u_i^{2m-2}z_i^2$ for $m\ge 1/2$, but
this requires certain growth conditions on the nonlinearities, which we wish
to avoid to simplify the presentation. The Lipschitz continuity of $A_{ij}$
is needed to estimate the difference $|A_{ij}(u_\sigma^k)-A_{ij}(u_K^k)|$
in the convergence proof. It is not needed to show the existence
of discrete solutions. The first bound in Hypothesis (H6) is a natural
growth condition needed in the entropy method, while the second bound
is used to estimate the discrete time derivative; 
see the proof of Lemma \ref{lem.est2}.

We introduce the discrete entropy
\begin{align}\label{2.def.entro}
  H[u^k] = \sum_{K\in\T}\m(K)h(u_K^k)\quad\mbox{for } k\ge 0.
\end{align}

\begin{theorem}[Existence of discrete solutions]\label{thm.ex}
Let Hypotheses (H1)--(H6) hold and let $\Delta t$ $<1/C_f$. Then there exists
a solution $u^k=(u_1^k,\ldots,u_n^k)\in\mathcal{H}_\T^n$ to scheme
\eqref{2.init}--\eqref{2.usigma} satisfying $u_{i,K}^k\ge 0$ for all $K\in\T$, 
$k \geq 1$, and $i=1,\ldots,n$, and it holds that
\begin{equation}\label{2.ei}
  (1-C_f\Delta t)H[u^k] + c_A\Delta t\sum_{i=1}^n\sum_{\sigma\in\E}\tau_\sigma
	(\mathrm{D}_\sigma u_i^k)^2 \le H[u^{k-1}] + C_f\Delta t\m(\Omega), \quad k\ge 1.
\end{equation}
\end{theorem}

The proof of Theorem \ref{thm.ex} is based on a topological degree argument. 
For this, we linearize and ``regularize'' scheme \eqref{2.init}--\eqref{2.usigma}.
The regularization is needed since we are working in the entropy variables
$w_i=h'_i(u_i)$ and the diffusion operator in these variables is only
positive semidefinite. 
Then we establish an entropy inequality associated to the approximate scheme 
and perform the limit when the regularization parameter vanishes.

For the convergence result, we need some notation. For $K\in\T$ and $\sigma\in\E_K$,
we define the cell $T_{K,\sigma}$ of the dual mesh:
\begin{itemize}
\item If $\sigma=K|L\in\E_{{\rm int},K}$, then $T_{K,\sigma}$ is that cell 
(``diamond'') whose
vertices are given by $x_K$, $x_L$, and the end points of the edge $\sigma$.
\item If $\sigma\in\E_{{\rm ext},K}$, then $T_{K,\sigma}$ is that cell (``triangle'')
whose vertices are given by $x_K$ and the end points of the edge $\sigma$.
\end{itemize}
The cells $T_{K,\sigma}$ define a partition of $\Omega$. It follows from the
property that the straight line $\overline{x_Kx_L}$ between two neighboring 
centers of cells is orthogonal to the edge $\sigma=K|L$ that
\begin{equation*}
  \m(\sigma)\dist(x_K,x_L) = 2\m(T_{K,\sigma}) \quad\mbox{for }
	\sigma=K|L\in \E_{\rm int}.
\end{equation*}
The approximate gradient of $v\in \mathcal{H}_{\T,\Delta t}$ is then defined by
$$
  \na^{\mathcal D} v(x,t) = \frac{\m(\sigma)}{\m(T_{K,\sigma})}
	(\mathrm{D}_{K,\sigma} v^k)\nu_{K,\sigma}
	\quad\mbox{for }x\in T_{K,\sigma},\ t\in(t_{k-1},t_{k}],
$$
where $\nu_{K,\sigma}$ is the unit vector that is normal to $\sigma$ and points
outwards of $K$.

We introduce a family $(\mathcal{D}_m)_{m\in\N}$ of admissible space-time 
discretizations of $\Omega_T$ indexed by the size 
$\eta_m=\max\{\Delta x_m,\Delta t_m\}$ of the mesh, satisfying $\eta_m\to 0$
as $m\to\infty$. We denote by $\T_m$ the corresponding meshes of $\Omega$ and
by $\Delta t_m$ the corresponding time step sizes.  
Finally, we set $\na^m:=\na^{\mathcal{D}_m}$.

\begin{theorem}[Convergence of the scheme]\label{thm.conv}
Let the assumptions of Theorem \ref{thm.ex} hold, let $(\mathcal{D}_m)_{m\in\N}$ 
be a family of admissible meshes satisfying \eqref{2.dd} uniformly in $m\in\N$,
and assume that $\Delta t_m< 1/C_f$ for $m\in\N$. Let $(u_m)_{m\in\N}$ be a family
of finite-volume solutions to \eqref{2.init}--\eqref{2.usigma} constructed in
Theorem \ref{thm.ex}. Then there exists a function $u=(u_1,\ldots,u_n)\in
L^2(0,T;H^1(\Omega;\R^n))$ satisfying $u_i\ge 0$ in $\Omega_T$, $i=1,\ldots,n$,
\begin{align*}
  u_{i,m}\to u_i &\quad\mbox{strongly in }L^2(\Omega_T), \\
	\na^m u_{i,m}\rightharpoonup\na u_i &\quad\mbox{weakly in }L^2(\Omega_T),
	\mbox{ as }m\to\infty,
\end{align*}
up to a subsequence,
and $u$ is a weak solution to \eqref{1.eq} and \eqref{1.bic}, i.e., for all
$\psi_i\in C_0^\infty(\Omega\times[0,T))$, it holds that
\begin{align*}
  \int_0^T&\int_\Omega u_i\pa_t\psi_i dxdt + \int_\Omega u_i^0(0)\psi_i(0)dx \\
	&= \int_0^T\int_\Omega\sum_{j=1}^n A_{ij}(u)\na u_j\cdot\na\psi_i dxdt
	+ \int_0^T\int_\Omega f_i(u)\psi_i dxdt, \quad i=1,\ldots,n.
\end{align*}
\end{theorem}

The proof is based on suitable estimates uniform with respect to $\Delta x$ and 
$\Delta t$, derived from the entropy inequality \eqref{2.ei} and
the discrete Gagliardo--Nirenberg inequality, as well as a
version of the Aubin--Lions lemma obtained in \cite{GaLa12}. This yields the
a.e.\ convergence of a sequence $(u_m)$ of solutions to scheme 
\eqref{2.init}--\eqref{2.usigma}. The final step is the identification of
the limit function as a weak solution to \eqref{1.eq} and \eqref{1.bic}.

The last result is the convergence of the discrete solutions, as $k\to\infty$,
to a constant stationary solution when the source terms vanish. For this, let 
$\bar u_i = \m(\Omega)^{-1}\int_\Omega u_i^0dx$ for $i=1,\ldots,n$
and $\bar u=(\bar u_1,\ldots,\bar u_n)$. We introduce for every $k\ge 1$ the
discrete relative entropy 
$$
  H[u^k|\bar u] = \sum_{i=1}^n\sum_{K\in\T}\m(K)
  h_i\bigg(\frac{u_i^k}{\bar u_i}\bigg)\bar u_i.
$$
Observe that since $H[u^k|\bar u]$ distinguishes from $H[u^k]$ only by
linear terms, so the entropy inequality \eqref{2.ei} also holds for the
relative entropy.

\begin{theorem}[Discrete large-time asymptotics]\label{thm.time}
Let $u^k\in\mathcal{H}_\T^n$ be a finite-volume solution to 
\eqref{2.init}--\eqref{2.usigma} for $k\ge 1$. Then
$$
  \sum_{i=1}^n \|u_i^k-\bar u_i\|^2_{0,2,\T}\to 0\quad\mbox{as }k\to\infty.
$$
If the entropy density is defined by \eqref{1.h} and the entropy inequality
is given by
\begin{equation}\label{2.ei2}
  H[u^k|\bar u] + \Delta t\sum_{i=1}^n\sum_{\sigma\in\E}\tau_\sigma
  \big(c_A (\mathrm{D}_\sigma u^k_i)^2 
  + 4 c_A'(\mathrm{D}_\sigma (u^k_i)^{1/2})^2 \big)
  \le H[u^{k-1}|\bar u],
\end{equation}
where $c'_A>0$, then there exist constants $\kappa>0$ (depending on $u^0$) and
 $\lambda>0$ (depending on $c'_A$, $u^0$, and $\zeta$) such that
$$
  \sum_{i=1}^n \pi_i \|u_i^k-\bar u_i\|^2_{0,1,\T} 
  \le \kappa H[u^0|\bar u]e^{-\lambda t_k}
  \quad\mbox{for all }k\ge 1.
$$
\end{theorem}

The proof of Theorem \ref{thm.time} is based on the entropy inequality 
\eqref{2.ei2} and some discrete functional inequalities and is rather standard. 
Inequality \eqref{2.ei2} follows if we assume that the matrix $A(u)$ 
satisfies for all $z \in \R^n$ and $u\in(0,\infty)^n$,
\begin{equation}\label{2.assum.A}
  z^\top h''(u) A(u) z \geq c_A |z|^2 + c'_{A} \sum_{i=1}^n \frac{z_i^2}{u_i}.
\end{equation}
This can be seen by slightly modifying the proof of Theorem \ref{thm.ex};
see Remark \ref{rem.ei2}.
All the assumptions of the theorem are fulfilled by the SKT model if
$a_{i0}>0$ for all $i=1,\ldots,n$ \cite[Lemma 4, Lemma 6]{CDJ18}.
When the Lotka--Volterra terms do not vanish, nonconstant steady states
are possible, and we present some numerical illustrations in this direction
in Section \ref{ssec.pattern}.


\section{Examples}\label{sec.exam}

We present several examples for which Hypotheses (H4)--(H6) are satisfied.
The examples include the SKT model.

\subsection{The $n$-species SKT cross-diffusion system}

Consider system \eqref{1.eq}--\eqref{1.bic}. The entropy density defined by
\eqref{1.h} satisfies Hypothesis (H4). Hypothesis (H5) is satisfied 
if $a_{ii}>0$ for all $i=1,\ldots,n$ and the detailed-balance condition 
\begin{equation}\label{skt1}
  \pi_i a_{ij} = \pi_j a_{ji} \quad\mbox{for all }i\neq j,
\end{equation}
holds, or if self-diffusion dominates cross-diffusion in the sense
\begin{equation}\label{skt2}
  \eta_0 := \min_{i=1,\ldots,n}\bigg(a_{ii}-\frac14\sum_{j=1}^n
	\big(\sqrt{a_{ij}}-\sqrt{a_{ji}}\big)^2\bigg) > 0
\end{equation}
and $\pi_i=1$ for $i=1,\ldots,n$; see Lemmas 4 and 6 in \cite{CDJ18}. 
In the former case, $c_A=\min_i \pi_ia_{ii}>0$ and in the latter case, $c_A=2\eta_0>0$.
The Lotka--Volterra source terms \eqref{1.f} satify Hypothesis (H6) with 
$C_f$ given by
\begin{equation}\label{skt.defCf}
  C_f = \frac{2}{\log 2}\max_{i=1,\ldots,n} \bigg(b_{i0}
	+ \frac{1}{e\pi_i}\sum_{j=1}^n \pi_j b_{ji} \bigg)
\end{equation}
(see Appendix \ref{app} for a proof). The existence of a constant $C'_f > 0$ 
such that
$$
\sum_{i=1}^n|f_i(u)|\le C'_f\bigg(1+\sum_{i=1}^n|u_i|^2\bigg)
$$
is clear since $f_i$ is growing at most as $u_j^2$.
This shows that Hypotheses (H4)--(H6) are fulfilled, 
and we have the following result.

\begin{corollary}
Let $a_{ii}>0$, $b_{ii}>0$ for $i=1,\ldots,n$ and let the diffusion matrix,
source terms, and entropy density be defined by \eqref{1.A}, \eqref{1.f}, and
\eqref{1.h}, respectively. We assume that \eqref{skt1} or \eqref{skt2} holds and
that $\Delta t<1/C_f$. Then there exists a finite-volume solution to scheme
\eqref{2.init}--\eqref{2.usigma} satisfying \eqref{2.ei}. Under the assumptions
of Theorem \ref{thm.conv}, the solutions associated to the meshes $(\mathcal{D}_m)$
converge to a solution to \eqref{1.eq}--\eqref{1.bic}, up to a subsequence.
\end{corollary}


\subsection{A cross-diffusion system for fluid mixtures}

It was shown in \cite{CDJ19} that the mean-field limit in a stochastic interacting
particle system leads to the cross-diffusion system \eqref{1.eq}
with diffusion coefficients
$$
  A_{ij} = \delta_{ij} a_{i0} + a_{ij}u_i, \quad
	i,j=1,\ldots,n,
$$
where $a_{ij}\ge 0$, and with vanishing source terms, $f_i=0$. 
This system is similar to the $n$-species SKT-type model,
but the diagonal diffusion is smaller. 
We choose the entropy density \eqref{1.h} with $\pi_i$ 
satisfying the detailed-balance condition \eqref{skt1} and assume that the
matrix $(\pi_i a_{ij})$ is positive definite with smallest eigenvalue $\lambda_0>0$.
Then Hypothesis (H5) is satisfied since
$$
  z^\top h''(u)A(u)z = \sum_{i=1}^n\frac{\pi_i}{u_i}z_i^2
	+ \sum_{i,j=1}^n(\pi_i a_{ij})z_iz_j \ge \lambda_0|z|^2,
$$
for all $z=(z_1,\dots,z_n)\in\R^n$. Thus, Hypotheses (H4)--(H5) are fulfilled.
A finite-volume scheme for this system has been already analyzed
in \cite{JuZu20}. However, the design and the analysis of the scheme are based 
on a weighted quadratic entropy, i.e.\ an entropy not of the form 
$\sum_{i=1}^n h_i(u_i)$.


\subsection{A cross-diffusion model for seawater intrusion}

The seawater intrusion model analyzed in \cite{Oul18} describes the evolution
of the height $u_1$ of freshwater and the height $u_2$ of saltwater in
a porous medium. The asymptotic limit of vanishing aspect ratio between
the thickness and the horizontal length of the porous medium in a Darcy transport
model leads to the cross-diffusion system \eqref{1.eq} with diffusion coefficients
$$
  A(u) = \begin{pmatrix}
	\delta u_1 & \delta u_1 \\ \delta u_2 & u_2
	\end{pmatrix},
$$
where $\delta\in(0,1)$ is the ratio of the freshwater and saltwater density, 
and with no source terms.
The original model contains a variable bottom $b(x)$ of the porous medium;
we assume for simplicity that the bottom is flat, $b(x)=0$. Our arguments
also hold for nonconstant functions $b(x)$ if $\na b\in L^\infty(\Omega)$.
The entropy density is given by
$$
  h(u) = \frac{1}{\delta}\big(u_1(\log u_1-1)+1\big) 
	+ \big(u_2(\log u_2-1)+1\big),
$$
and a computation shows that 
$$
  z^\top h''(u)A(u)z = \frac12(1-\delta)(z_1^2+z_2^2) + \frac12(1+\delta)(z_1+z_2)^2
  \ge \frac12(1-\delta)|z|^2,
$$
for $z\in\R^2$. We infer that Hypotheses (H4)--(H5) are fulfilled.

An entropy-dissipating finite-volume scheme, based on a two-point approximation 
with upwind mobilities, was already suggested and
analyzed in \cite{Oul18} using similar techniques as in our paper. However,
our analysis allows us to recast this model in a more general framework.


\subsection{A Keller--Segel system with additional cross-diffusion}

It is well known that the parabolic-parabolic Keller--Segel model may lead
to finite-time blow-up of weak solutions \cite{BDP06}. Adding cross-diffusion
in the equation for the chemical signal allows for global weak solutions,
which may help to approximate the Keller--Segel system close to the blow-up time.
The evolution of the cell density $u_1$ and the chemical concentration $u_2$ is
governed by equations \eqref{1.eq} in two space dimensions with $f_1(u)=0$ and 
$f_2(u)=u_1-u_2$ and with the diffusion matrix 
(take $m=2$ and $n=1$ in \cite{CHJ12})
$$
  A(u) = \begin{pmatrix} 2u_1 & -u_1 \\ \delta & 1 \end{pmatrix},
$$
where $\delta>0$ describes the strength of cross-diffusion (and
can be arbitrarily small). The associated entropy density
given by
$$
  h(u) = h_1(u_1) + h_2(u_2) = \big(u_1(\log u_1-1)+1\big) + \frac{1}{2\delta}u_2^2
$$
does {\em not} satisfy Hypothesis (H4), since $h'_2:(0,\infty)\to\R$,
$h'_2(u_2)=u_2/\delta$, is not invertible, but it satisfies Hypothesis (H5):
$$
  z^\top h''(u)A(u) = 2z_1^2 + \delta^{-1}z_2^2 \ge \min\{2,\delta^{-1}\}|z|^2
		\quad\mbox{for }z\in\R^2.
$$
Hypothesis (H6) is satisfied since the elementary inequalities 
$u_1\log u_2\le u_1\log u_1 - u_1 + u_2$ for $u_1$, $u_2>0$ and $-u_2\log u_2
\le e^{-1}$ imply that $f_2(u)\log u_2\le C(1+h(u))$.
Although, formally, we cannot apply the results of the previous
section, the technique still applies by defining $h_2'$ as a function from $\R$ 
to $\R$. We note that $u_2$ cannot be proven to be nonnegative, even not on
the continuous level. However, the concentration $u_2$ 
becomes nonnegative in the limit $\delta\to 0$.


\section{Proof of Theorem \ref{thm.ex}}\label{sec.ex}

We prove Theorem \ref{thm.ex} by induction. If $k=0$, we have $u_{i,K}^0\ge 0$
for all $K\in\T$, $i=1,\ldots,n$ by assumption (H3). 
Assume that there exists a solution $u^{k-1}$
to \eqref{2.fvm}--\eqref{2.usigma} satisfying $u_{i,K}^{k-1}\ge 0$ for
$K\in\T$, $i=1,\ldots,n$. The construction of $u^k$ is split into several steps.

{\em Step 1: Definition of a linearized problem.} Let $R>0$ and $\eps>0$. We define
the set
$$
  Z_R = \big\{w=(w_{1},\ldots, w_{n}) \in \mathcal{H}_\T^n : 
	\|w_{i}\|_{1,2,\T}<R\mbox{ for }i=1,\ldots,n\big\},
$$
and the mapping $F_\eps:Z_R\to\R^{n\theta}$, $F_\eps(w)=w^\eps$, with $\theta=\#\T$
and $w^\eps=(w^\eps_1,\ldots,w^\eps_n)$ is the solution to the linear problem
\begin{equation}\label{3.lin}
  \eps\sum_{\sigma\in\E_K}\tau_\sigma \text{D}_{K,\sigma}w_i^\eps
	- \eps\m(K) w^\eps_{i,K}
  = \frac{\m(K)}{\Delta t}(u_{i,K}-u_{i,K}^{k-1})
	+ \sum_{\sigma\in\E_K}\mathcal{F}_{i,K,\sigma} - \m(K) f_i(u_K),
\end{equation}
for $K\in\T$ and $i=1,\ldots,n$, where $\mathcal{F}_{i,K,\sigma}$ is defined in
in \eqref{2.flux} and $u_{i,K}=(h_i')^{-1}(w_{i,K})>0$ is a function of $w$.
The existence of a unique solution $w^\eps$ to this problem is a consequence of
the proof of \cite[Lemma 9.2]{EGH00}. 

{\em Step 2: Continuity of $F_\eps$.} We fix $i\in\{1,\ldots,n\}$, multiply
\eqref{3.lin} by $w_{i,K}^\eps$, sum over $K\in\T$, and apply discrete integration
by parts:
\begin{align}
  \eps \|w^\eps_{i}\|^2_{1,2,\T}
	&= -\sum_{K \in \T}\frac{\m(K)}{\Delta t}(u_{i,K}-u_{i,K}^{k-1}) w^{\eps}_{i,K} 
	+ \sum_{\substack{\sigma\in\E_{\mathrm{int}} \\ \sigma=K|L}}\mathcal{F}_{i,K,\sigma} 
	\text{D}_{K,\sigma} w_{i}^\eps \nonumber \\
	&\phantom{xx}{}+ \sum_{K\in\T} \m(K) f_i(u_K) w^\eps_{i,K} 
	=: J_1 + J_2 + J_3. \label{3.aux1}
\end{align}
By the Cauchy--Schwarz inequality and definition \eqref{2.flux} of 
$\mathcal{F}_{i,K,\sigma}$, we find that
\begin{align*}
  J_1 &\le \frac{1}{\Delta t}\|u_i-u_i^{k-1}\|_{0,2,\T}\|w_i^\eps\|_{0,2,\T}, \\
	J_2 &\le \bigg(\sum_{j=1}^n\sum_{\sigma\in\E}\tau_\sigma A_{ij}(u_\sigma)
	(\text{D}_\sigma u_j)^2\bigg)^{1/2}\bigg(\sum_{\sigma\in\E}\tau_\sigma
	(\text{D}_\sigma w_i^\eps)^2\bigg)^{1/2}, \\
	J_3 &\le \|f_i(u)\|_{0,2,\T}\|w_i^\eps\|_{0,2,\T}.
\end{align*}
Since $w\in Z_R$ is bounded, so does $u\in\mathcal{H}_\T^n$. Thus, there exists
a constant $C(R)>0$ independent of $w^\eps$ such that $J_i\le C(R)\|w_i^\eps\|_{1,2,\T}$.
We deduce from \eqref{3.aux1} that
$$
  \eps\|w_i^\eps\|_{1,2,\T} \le C(R).
$$

We now turn to the proof of continuity of $F_\eps$. Let $(w^m)_{m\in\N}\in Z_R$
be such that $w^m\to w$ as $m\to\infty$. The previous bound shows that 
$w^{\eps,m}=F_\eps(w^m)$ is uniformly bounded. By the theorem of 
Bolzano--Weierstra{\ss}, there exists a subsequence of
$(w^{\eps,m})$, which is not relabeled, such that $w^{\eps,m}\to w^\eps$ as
$m\to\infty$. Passing to the limit in scheme \eqref{3.lin} and taking into account
the continuity of the nonlinear functions, we see that $w_i^\eps$ is a solution
to \eqref{3.lin} for all $i=1,\ldots,n$, and it holds that $w^\eps=F_\eps(w)$.
Because of the uniqueness of the limit function, the whole sequence converges,
which proves the continuity.

{\em Step 3: Existence of a fixed point.} We claim that $F_\eps$ admits a fixed point.
We use a topological degree argument \cite[Chap.~1]{Dei85}, i.e., we prove that 
$\operatorname{deg}(I-F_\eps,Z_R,0) = 1$, where $\operatorname{deg}$ is the 
Brouwer topological degree. 
Since $\operatorname{deg}$ is invariant by homotopy, it is sufficient to prove that 
any solution $(w^\eps,\rho)\in \overline{Z}_R\times[0,1]$ to the fixed-point equation 
$w^\eps = \rho F_\eps(w^\eps)$ satisfies $(w^\eps,\rho)\not\in\pa Z_R\times[0,1]$ 
for sufficiently large values of $R>0$. Let $(w^\eps,\rho)$ be a fixed point 
and $\rho\neq 0$ (the case $\rho=0$ is clear). Then $w^\eps$ solves
\begin{equation}\label{3.approx}
  \eps\sum_{\sigma\in\E_K}\tau_\sigma \text{D}_{K,\sigma}w_i^\eps 
	- \eps\m(K) w^\eps_{i,K}
  = \rho\bigg(\frac{\m(K)}{\Delta t}(u^\eps_{i,K}-u_{i,K}^{k-1})
	+ \sum_{\sigma\in\E_K}\mathcal{F}^\eps_{i,K,\sigma} - \m(K)f_i(u^\eps_{i,K})\bigg),
\end{equation}
for all $K\in\T$ and $i=1,\ldots,n$, where $u^\eps_{i,K}=(h_i')^{-1}(w^\eps_{i,K})>0$
and $\mathcal{F}_{i,K,\sigma}^\eps$
is defined as in \eqref{2.flux} with $u$ replaced by $u^\eps$.
The following discrete entropy inequality is the key argument.

\begin{lemma}[Discrete entropy inequality]\label{lem.ei}
Let the assumptions of Theorem \ref{thm.ex} hold, let $0<\rho\le 1$, $\eps>0$, and
let $u^\eps$ be a solution to \eqref{3.approx}. Then
\begin{align}
  \rho(1-C_f\Delta t)H[u^\eps] &+ \eps\Delta t\sum_{i=1}^n\|w_i^\eps\|_{1,2,\T}^2 
	\nonumber \\
	&{}+ \rho c_A\Delta t\sum_{i=1}^n\sum_{\sigma\in\E}\tau_\sigma
	(\mathrm{D}_\sigma u_i^\eps)^2 \le \rho H[u^{k-1}] + \rho C_f\Delta t\m(\Omega).
	\label{3.ei}
\end{align}
\end{lemma}

\begin{proof}
We multiply \eqref{3.approx} by $\Delta t w_{i,K}^\eps$ and sum over $i=1,\ldots,n$
and $K\in\T$. Then, after a discrete integration by parts, 
\begin{align*}
  & \eps\Delta t\sum_{i=1}^n \|w_{i}^\eps\|^2_{1,2,\T}
	+ J_4 + J_5 + J_6 = 0, \quad\mbox{where} \\
	& J_4 = \rho\sum_{i=1}^n\sum_{K\in\T}\m(K)(u_{i,K}^\eps-u_{i,K}^{k-1})w_{i,K}^\eps, \\
	& J_5 = \rho \Delta t \sum_{i,j=1}^n \sum_{\substack{\sigma\in\E_{\rm int} 
	\\ \sigma=K|L}} \tau_\sigma A_{ij}(u^\eps_\sigma)
	\text{D}_{K,\sigma} u_j^\eps \text{D}_{K,\sigma} w_{i}^\eps, \\
	& J_6 = \rho \Delta t \sum_{i=1}^n \sum_{K \in \T} \m(K) f_i(u^\eps_K) w^\eps_{i,K}.
\end{align*}
Since $h_i$ is assumed to be convex, we have
\begin{align*}
  J_4 &= \rho\sum_{i=1}^n\sum_{K\in\T}\m(K)(u_{i,K}^\eps-u_{i,K}^{k-1})
	h_i'(u_{i,K}^\eps) \\
	&\ge \rho\sum_{i=1}^n\sum_{K\in\T}\m(K)\big(h_i(u_{i,K}^\eps)-h_i(u_{i,K}^{k-1})\big)
	= \rho\big(H[u^\eps]-H[u^{k-1}]\big).
\end{align*}
We deduce from $w_i^\eps=h_i'(u_i^\eps)$, the discrete chain rule \eqref{2.chain}, 
and Hypothesis (H5) that
\begin{align*}
  J_5 &= \rho\Delta t\sum_{i,j=1}^n\sum_{\substack{\sigma\in\E_{\rm int} \\ \sigma=K|L}}
  \tau_\sigma A_{ij}(u_\sigma^\eps)\text{D}_{K,\sigma}(u_j^\eps)
	\text{D}h_i'(u_i^\eps) \\
	&= \rho\Delta t\sum_{i,j=1}^n\sum_{\substack{\sigma\in\E_{\rm int} \\ \sigma=K|L}}
  \tau_\sigma h''_i(u_{i,\sigma}^\eps) A_{ij}(u_\sigma^\eps)
	\text{D}_{K,\sigma}u_i^\eps\text{D}_{K,\sigma}u_j^\eps \\
  &\ge \rho c_A\Delta t\sum_{i=1}^n\sum_{\sigma\in\E}\tau_\sigma
	(\text{D}_\sigma u_i^\eps)^2.
\end{align*}
Finally, by Hypothesis (H6),
$$
  J_6 \ge -\rho C_f\Delta t\sum_{K\in\T}\m(K)(1+h(u_K^\eps))
	= -\rho C_f\Delta t H[u^\eps] - \rho C_f\Delta t\m(\Omega).
$$
This completes the proof.
\end{proof}

We proceed with the topological degree argument. Set
$$
  R = \frac{1}{\sqrt{\eps\Delta t}} 
	\big(H[u^{k-1}]+C_f \Delta t \m(\Omega)\big)^{1/2}+1.
$$
The previous lemma implies that
$$
  \eps\Delta t\sum_{i=1}^n\|w_i^\eps\|^2_{1,2,\T}
	\le \rho \big( H[u^{k-1}] +C_f \Delta t \m(\Omega) \big) 
	\le \eps\Delta t(R-1)^2,
$$
which gives $\sum_{i=1}^n\|w_i^\eps\|^2_{1,2,\T}<R^2$.
We conclude that $w^\eps \not\in\pa Z_R$ and $\operatorname{deg}(I-F_\eps,Z_R,0)=1$. 
Thus, $F_\eps$ admits at least one fixed point.

{\em Step 4: Limit $\eps\to 0$.}
We deduce from Hypothesis (H4), Lemma \ref{lem.ei}, and $\Delta t<1/C_f$ 
that for any $K\in\T$ and $i=1,\ldots,n$,
$$
  c_h \m(K)(u^\eps_{i,K}-1) \leq \m(K)h_i(u_{i,K}^\eps)
	\le H[u^\eps] \le  \frac{H[u^{k-1}] + C_f\Delta t\m(\Omega)}{1-C_f\Delta t}.
$$
This shows that $(u_{i,K}^\eps)$ is bounded uniformly in $\eps$. Therefore,
there exists a subsequence (not relabeled) such that $u_{i,K}^\eps\to u_{i,K}$
as $\eps\to 0$. Lemma \ref{lem.ei} implies the existence of a subsequence
such that $\eps w_{i,K}^\eps\to 0$. Hence, performing the limit $\eps\to 0$
in \eqref{3.approx}, we deduce the existence of a solution to 
\eqref{2.init}--\eqref{2.usigma}. Passing to the limit $\eps\to 0$ in \eqref{3.ei}
yields the entropy inequality \eqref{2.ei}, which finishes the proof of Theorem
\ref{thm.ex}.

\begin{remark}\label{rem.ei2}\rm
If we assume that \eqref{2.assum.A} holds then, arguing as for $J_5$ in the proof 
of Lemma \ref{lem.ei}, we obtain an additional term of the form
$$
  \rho c'_A \Delta t \sum_{i=1}^n \sum_{\sigma \in \E} \tau_\sigma \frac{(\mathrm{D}_\sigma u^\eps_i)^2}{u^\eps_{i,\sigma}}.
$$
This expression is well defined since it holds that $u^\eps_{i,K} > 0$ for all 
$K \in \T$ and $i=1,\ldots,n$ and consequently $u_{i,\sigma}^\eps>0$. We deduce
from the elementary inequality $(x-y)(\log x-\log y)\ge 4(\sqrt{x}-\sqrt{y})^2$
for all $x$, $y>0$ that
$$
  \rho c'_A \Delta t \sum_{i=1}^n \sum_{\sigma \in \E} \tau_\sigma 
  \frac{(\mathrm{D}_\sigma u^\eps_i)^2}{u^\eps_{i,\sigma}} 
  \geq 4 \rho c'_A \Delta t \sum_{i=1}^n \sum_{\sigma \in \E} \tau_\sigma 
  \big(\mathrm{D}_\sigma (u^\eps_i)^{1/2} \big)^2.
$$
Thus, if assumption \eqref{2.assum.A} holds, we conclude that for every $\eps > 0$,
\begin{align*}
  \rho(1-C_f\Delta t)H[u^\eps] &+ \eps\Delta t\sum_{i=1}^n\|w_i^\eps\|_{1,2,\T}^2 
  + \rho c_A\Delta t\sum_{i=1}^n\sum_{\sigma\in\E}\tau_\sigma 
  (\mathrm{D}_\sigma u_i^\eps)^2 \\
  &{}+ 4 \rho c'_A \Delta t \sum_{i=1}^n \sum_{\sigma \in \E} \tau_\sigma 
  \big(\mathrm{D}_\sigma (u^\eps_i)^{1/2} \big)^2 
  \le \rho H[u^{k-1}] + \rho C_f\Delta t\m(\Omega).
\end{align*}
Finally, applying similar arguments as at the end of the proof of Theorem 
\ref{thm.ex} and since the relative entropy $H[\cdot|\bar{u}]$ distinguishes 
from the entropy \eqref{2.def.entro} only by linear terms, we obtain the 
entropy inequality \eqref{2.ei2}.
\end{remark}


\section{A priori estimates}\label{sec.est}

We establish some a priori estimates uniform in $\Delta x$ and $\Delta t$
for the solutions to \eqref{2.init}--\eqref{2.usigma}.

\begin{lemma}[Discrete space estimates]\label{lem.est1}
Let the assumptions of Theorem \ref{thm.ex} hold and let $\Delta t<1/C_f$.
Then there exists a constant $C>0$ independent of $\Delta x$ and $\Delta t$
such that for $i=1,\ldots,n$,
$$
  \max_{k=1,\ldots,N_T}\|u_i^k\|_{0,1,\T}
	+ \sum_{k=1}^{N_T}\Delta t\|u_i^k\|_{1,2,\T}^2
	+ \sum_{k=1}^{N_T}\Delta t\|u_i^k\|_{0,3,\T}^3 \le C.
$$
\end{lemma}

\begin{proof}
Let $i\in\{1,\ldots,n\}$ be fixed. After summing \eqref{2.ei} over $K\in\T$ and
applying the discrete Gronwall inequality, Hypothesis (H4) shows that
$$
  \max_{k=1,\ldots,N_T}\|u_i^k\|_{0,1,\T}
	+ \sum_{k=1}^{N_T}\Delta t|u_i^k|_{1,2,\T}^2 \le C.
$$
By the discrete Poincar\'e--Wirtinger inequality \cite[Theorem 3.6]{BCF15},
we infer the bound 
$$
  \sum_{k=1}^{N_T}\Delta t\|u_i^k\|^2_{0,2,\T}\le C.
$$
Consequently, $\sum_{k=1}^{N_T}\Delta t\|u_i^k\|_{1,2,\T}^2 \le C$.
In order to show the remaining bound, we apply the discrete Gagliardo--Nirenberg
inequality with $\theta=2/3$ \cite[Theorem 3.4]{BCF15}:
$$
  \|u_i^k\|_{0,3,\T} \le C\zeta^{-\theta/2}\|u_i^k\|_{1,2,\T}^\theta
	\|u_i^k\|_{0,1,\T}^{1-\theta} 
	\le C\zeta^{-1/3}\Big(\max_{\ell=1,\ldots,N_T}\|u_i^\ell\|_{0,1,\T}^{1/3}\Big)
	\|u_i^k\|_{1,2,\T}^{2/3}.
$$
Summing over $k=1,\ldots,N_T$ gives
$$
  \sum_{k=1}^{N_T}\Delta t\|u_i^k\|_{0,3,\T}^3 \le C\zeta^{-1}
	\max_{\ell=1,\ldots,N_T}\|u_i^\ell\|_{0,1,\T}
	\sum_{k=1}^{N_T}\Delta t\|u_i^k\|_{1,2,\T}^2
  \le C.
$$
This ends the proof.
\end{proof}

The the previous proof, we use the fact that the domain $\Omega$ is two-dimensional.
We can derive a uniform estimate for $u_i^k$ in $L^{8/3}(\Omega_T)$ 
in three-dimensional domains.
This bound is sufficient subject to an adaption of the space for the following estimate.
Let the discrete time derivative 
of a function $v\in\mathcal{H}_{\T,\Delta t}$ be given by
$$
  \pa_t^{\Delta t}v^k = \frac{v^k-v^{k-1}}{\Delta t}, \quad k=1,\ldots,N_T.
$$

\begin{lemma}[Discrete time estimate]\label{lem.est2}
Let the assumptions of Theorem \ref{thm.ex} hold and let $\Delta t<1/C_f$.
Then there exists a constant $C>0$ independent of $\Delta x$ and $\Delta t$
such that for $i=1,\ldots,n$,
$$
  \sum_{k=1}^{N_T}\Delta t\|\pa_t^{\Delta t}u_i^k\|_{-1,6/5,\T} \le C.
$$
\end{lemma}

\begin{proof}
Let $k\in\{1,\ldots,N_T\}$ and $i\in\{1,\ldots,n\}$ be fixed and let $\phi\in
\mathcal{H}_\T$ be such that $\|\phi\|_{1,6,\T}=1$. We multiply \eqref{2.fvm} by
$\phi_K$, sum over $K\in\T$, and apply discrete integration by parts:
\begin{align}
  \sum_{K\in\T} \m(K)\frac{u^k_{i,K}-u^{k-1}_{i,K}}{\Delta t} \phi_K 
	&= -\sum_{j=1}^n \sum_{\substack{\sigma\in\E_{{\rm int}} \\ \sigma=K|L}} 
	\tau_\sigma A_{ij}(u^k_\sigma) \text{D}_{K,\sigma} u^k_j\text{D}_{K,\sigma} \phi 
	+ \sum_{K \in \T} \m(K) f_i(u^k_K) \phi_K \nonumber \\
  &= J_7 + J_8. \label{3.J78}
\end{align}
The H\"older inequality and definition of $\tau_\sigma$ imply that
\begin{equation}\label{3.J7}
  |J_7| \le \sum_{j=1}^n\bigg(\sum_{\substack{\sigma\in\E_{{\rm int}} \\ \sigma=K|L}} 
	\m(\sigma)\dist_\sigma| A_{ij}(u_\sigma^k)|^3\bigg)^{1/3}
	|u_j^k|_{1,2,\T}|\phi|_{1,6,\T}.
\end{equation}
If $u_{i,K}^k\neq u_{i,K,\sigma}^k$, we have $u_{i,\sigma}^k=\widetilde u_{i,\sigma}^k$
and $\widetilde u_{i,\sigma}^k$ solves \eqref{2.chain}. 
Then Hypothesis (H4) implies \eqref{2.est.usigma} and 
in particular $0\le u_{i,\sigma}^k\le  u_{i,K}^k+u_{i,L}^k$.
By Hypothesis (H5), the diffusion coefficients $A_{ij}$ grow at most linearly.
Consequently, for $\sigma=K|L$,
$$
  |A_{ij}(u_\sigma^k)|^3 \le C\sum_{\ell=1}^n\big(1 + |u_{\ell,K}^k|^3 
	+ |u_{\ell,L}^k|^3\big).
$$
Hence, taking into account the mesh regularity \eqref{2.dd},
$$
  \sum_{\substack{\sigma\in\E_{{\rm int}} \\ \sigma=K|L}} 
	\m(\sigma)\dist_\sigma| A_{ij}(u_\sigma^k)|^3
	\le C\sum_{K\in\T}(1 + |u_{i,K}^k|^3)
	\sum_{\sigma\in\Eint}\m(\sigma)\zeta^{-1}\dist(x_K,\sigma).
$$
Using the property \cite[(1.10)]{EGH08}
\begin{equation*}
  \sum_{K\in\T}\sum_{\sigma\in\Eint}\m(\sigma)\dist(x_K,\sigma)
	\le 2\sum_{K\in\T}\m(K) = 2\m(\Omega)
\end{equation*}
(the constant on the right-hand side slightly changes in three space dimensions),
we conclude from \eqref{3.J7} that
$$
  |J_7|\le C\zeta^{-1}\bigg(1+\sum_{j=1}^n\|u_j^k\|^3_{0,3,\T}\bigg)^{1/3}
	\sum_{j=1}^n|u_j^k|_{1,2,\T}|\phi|_{1,6,\T}.
$$
Next, in view of Hypothesis (H6),
$$
  |J_8| \le \sum_{K\in\T}\m(K)|f_i(u_K^k)|\,|\phi_K|	
	\le C'_f\bigg(\|\phi\|_{0,1,\T} + \sum_{i=1}^n \|(u_i^k)^2 \phi\|_{0,1,\T}\bigg).
$$
We apply H\"older's inequality to conclude that there exists a constant 
$C>0$ independent of $\Delta x$ and $\Delta t$ such that
$$
  |J_8| \le C \bigg(\|\phi\|_{0,6,\T} + \sum_{i=1}^n\|u^k_i\|^2_{0,3,\T} 
	\|\phi\|_{0,3,\T} \bigg).
$$
Moreover, thanks to the discrete Poincar\'e-Sobolev inequality obtained 
in \cite[Theorem 3]{BCF15}, we have
$\|\phi\|_{0,3,\T} \le C \zeta^{-5/6}\|\phi\|_{1,6,\T}$,
which implies the existence of a constant, still denoted by $C>0$, such that
$$
  |J_8| \le C \bigg(1+\sum_{i=1}^n\|u^k_i\|^2_{0,3,\T}\bigg)\|\phi\|_{1,6,\T}.
$$
Inserting the estimates for $J_7$ and $J_8$ into \eqref{3.J78} and using
Lemma \ref{lem.est1} gives
\begin{align*}
  \sum_{k=1}^{N_T}&\Delta t\bigg\|\frac{u_i^k-u_i^{k-1}}{\Delta t}\bigg\|_{-1,6/5,\T}
	= \sup_{\|\phi\|_{1,6,\T}=1}\sum_{k=1}^{N_T}\Delta t\bigg|\sum_{K\in\T}\m(K)
	\frac{u_i^k-u_i^{k-1}}{\Delta t}\phi_K\bigg| \\
	&\le C\bigg(\sum_{k=1}^{N_T}\Delta t\bigg(1+\sum_{j=1}^n\|u_j^k\|_{0,3,\T}^3
	\bigg)\bigg)^{1/3}\bigg(\sum_{k=1}^{N_T}\Delta t\sum_{j=1}^n\|u_j^k\|_{1,2,\T}^2
	\bigg)^{1/2} \\
	&\phantom{xx}{}+ CT + C \sum_{i=1}^n \sum_{k=1}^{N_T}\Delta t 
	\|u_i^k\|_{0,3,\T}^2 \\
	&\le C\bigg(\sum_{k=1}^{N_T}\Delta t\bigg(1+\sum_{j=1}^n\|u_j^k\|_{0,3,\T}^3
	\bigg)\bigg)^{1/3}\bigg(\sum_{k=1}^{N_T}\Delta t\sum_{j=1}^n\|u_j^k\|_{1,2,\T}^2
	\bigg)^{1/2} \\
	&\phantom{xx}{}+ CT + C T^{1/3}\sum_{i=1}^n 
	\bigg(\sum_{k=1}^{N_T}\Delta t \|u_i^k\|_{0,3,\T}^3\bigg)^{2/3} \le C.
\end{align*}
This concludes the proof.
\end{proof}


\section{proof of Theorem \ref{thm.conv}}\label{sec.conv}

Before we prove the theorem, we show some compactness properties.

\subsection{Compactness properties}

Let $(\mathcal{D}_m)_{m\in\N}$ be a sequence of admissible meshes of $\Omega_T$
satisfying the mesh regularity \eqref{2.dd} uniformly in $m\in\N$ and let
$\Delta t_m<1/C_f$.
We claim that the estimates from Lemmas \ref{lem.est1} and \ref{lem.est2}
imply the strong convergence of a subsequence of $(u_{i,m})$.

\begin{proposition}[Strong convergence]\label{prop.conv}\sloppy
Let the assumptions of Theorem \ref{thm.conv} hold and let $(u_m)_{m\in\N}$ be
a sequence of discrete solutions to \eqref{2.init}--\eqref{2.usigma} constructed
in Theorem \ref{thm.ex}. Then there exists a subsequence of $(u_m)$, which is not
relabeled, and $u=(u_1,\ldots,u_n)\in L^3(\Omega_T)$ such that for any $p<3$
and $i=1,\ldots,n$,
$$
  u_{i,m}\to u_i\quad\mbox{strongly in }L^p(\Omega_T)\mbox{ as }m\to\infty.
$$
\end{proposition}

\begin{proof}
The idea is to apply the discrete version of the Aubin--Lions lemma obtained in
\cite[Theorem 3.4]{GaLa12}. Because of the estimates
$$
  \sum_{k=1}^{N_T}\Delta t\|u_i^k\|_{1,2,\T}^2 
	+ \sum_{k=1}^{N_T}\Delta t\|\pa_t^{\Delta t}u_i^k\|_{-1,6/5,\T} \le C,
$$
it remains to show that the discrete norms $\|\cdot\|_{1,2,\T}$ and 
$\|\cdot\|_{-1,6/5,\T}$ verify the following assumptions:
\begin{itemize}
\item[(1)] For any sequence $(v_m)_{m \in \N} \subset \mathcal{H}_{\T_m}$ such that 
there exists $C>0$ with $\|v_m\|_{1,2,\T_m} \le C$ for all $m \in \N$, 
there exists $v \in L^2(\Omega)$ satisfying, up to a subsequence, $v_m \to v$ 
strongly in $L^2(\Omega)$.
\item[(2)] If $v_m \to v$ strongly in $L^2(\Omega)$ and $\|v_m\|_{-1,6/5,\T_m} \to 0$ 
as $m \to \infty$, then $v=0$.
\end{itemize}
Property (1) is a direct consequence of \cite[Lemma 5.6]{EGH08}. For property (2), 
let $\phi\in C_0^\infty(\Omega)$ and set $\phi_K=\phi(x_K)$ for $K\in\T_m$ and
$$
  \phi_m(x) = \sum_{K\in\T_m}\phi_K\mathbf{1}_K(x), \quad x\in\Omega.
$$
Then $\phi_m\in\mathcal{H}_{\T_m}$ and, in view of the definition of 
$\|\cdot \|_{1,6,\T_m}$,
$$
  \|\phi \|^6_{1,6,\T_m} \le \m(\Omega) \|\phi\|^6_{\infty} 
	+ \|\na\phi\|^6_{L^\infty(\Omega)} \sum_{K\in \T} \sum_{\sigma\in \Eint} 
	\m(\sigma) \dist_\sigma.
$$
Hence, using \cite[(1.10)]{EGH08} this implies that
$\|\phi_m\|_{1,6,\T_m}\le C(\zeta,\Omega)
\|\phi\|_{W^{1,\infty}(\Omega)}$ and consequently,
\begin{equation}\label{3.aux2}
  \bigg|\int_\Omega v_m(x)\phi_m(x)dx\bigg| \le C(\Omega,\zeta)\|v_m\|_{-1,6/5,\T_m}
	\|\phi\|_{W^{1,\infty}(\Omega)}.
\end{equation}
Now, if we assume that $v_m\to v$ strongly in $L^2(\Omega)$ as $m\to\infty$, we have
$$
  \int_\Omega v_m(x)\phi_m(x)dx \to \int_\Omega v(v)\phi(x)dx,
$$
since also $\phi_m\to\phi$ strongly in $L^2(\Omega)$. Hence, if
$\|v_m\|_{-1,6/5,\T_m} \to 0$, we deduce from \eqref{3.aux2} that
$\int_\Omega v(x)\phi(x)dx = 0$, which yields $v=0$. This proves property (2).

We conclude from \cite[Theorem 3.4]{GaLa12} that, up to a subsequence,
$u_m\to u$ strongly in $L^1(0,T;L^2(\Omega))$. Then the uniform $L^3(\Omega_T)$ 
bound obtained in Lemma \ref{lem.est1} and the dominated convergence theorem show 
that $u_m\to u$ strongly in $L^p(\Omega_T)$ for any $p<3$.
\end{proof}

\begin{lemma}[Convergence of the gradient]\label{lem.grad}
Under the assumptions of Proposition \ref{prop.conv}, there exists a subsequence
of $(u_m)_{m\in\N}$ such that, as $m\to\infty$,
$$
  \na^m u_{i,m}\rightharpoonup\na u_i\quad\mbox{weakly in }L^2(\Omega_T),
	\quad i=1,\ldots,n,
$$
where $\na^m$ is defined in Section \ref{ssec.main}.
\end{lemma}

\begin{proof}
Lemma \ref{lem.est1} implies that $(\na^m u_{i,m})$ is bounded in $L^2(\Omega_T)$.
Thus, for a subsequence, $\na^m u_{i,m}\rightharpoonup v_i$ weakly in $\Omega_T$
as $m\to\infty$. It is shown in \cite[Lemma 4.4]{CLP03} that $v_i=\na u_i$.
\end{proof}


\subsection{Convergence of the scheme}

To finish the proof of Theorem \ref{thm.conv}, we need to show that the function
$u$ obtained in Proposition \ref{prop.conv} is a weak solution to 
\eqref{1.eq} and \eqref{1.bic}. To this end, we follow the strategy of \cite{CLP03}.
Let $i\in\{1,\ldots,n\}$ be fixed, let $\psi_i\in C_0^\infty(\Omega\times[0,T))$
be given, and let $\eta_m=\max\{\Delta x_m,\Delta t_m\}$ be sufficiently small
such that $\operatorname{supp}(\psi_i) \subset \{x\in\Omega:\dist(x,\pa\Omega)
> \eta_m\}\times[0,T)$.
For the limit, we introduce the following notation:
\begin{align*}
	F_{10}^m &= -\int_0^T\int_{\Omega}u_{i,m} \pa_t\psi_i dxdt
	- \int_{\Omega}u_{i,m}(x,0)\psi_i(x,0)dx, \\
	F_{20}^m &= \sum_{j=1}^n \int_0^T\int_{\Omega} A_{ij}(u_{m}) \na^m u_{j,m}\cdot 
	\na\psi_i dxdt, \\
	F_{30}^m &= -\int_0^T \int_{\Omega} f_i(u_{m}) \psi_i dxdt.
\end{align*}
The convergence results of Proposition \ref{prop.conv} and Lemma \ref{lem.grad},
the continuity of $A_{ij}$ and $f_i$, and the assumption on the initial data
show that, as $m\to\infty$,
\begin{align*}
  F_{10}^m+F_{20}^m+F_{30}^m
	&\to -\int_0^T\int_\Omega u_i\pa_t\psi_i dxdt
	- \int_\Omega u_i^0(x) \psi_i(x,0)dx \\	
	&\phantom{nm}{}+ \sum_{j=1}^n \int_0^T\int_\Omega A_{ij}(u) \na u_j \cdot 
	\na \psi_i dxdt - \int_0^T \int_\Omega f_i(u) \psi_i dxdt. 
\end{align*}

We proceed with the limit $m\to\infty$ in scheme \eqref{2.fvm}. For this, 
we set $\psi_{i,K}^k:=\psi_i(x_K,t_k)$, multiply \eqref{2.fvm} by 
$\Delta t_m\psi_{i,K}^{k-1}$, and sum over $K\in\T_m$ and $i=1,\ldots,n$, leading to
\begin{align}
  & F_1^m + F_2^m + F_3^m = 0, \quad\mbox{where} \label{4.FFF} \\
	& F_1^m = \sum_{k=1}^{N_T}\sum_{K\in\T}\m(K)\big(u_{i,K}^k-u_{i,K}^{k-1}\big)
	\psi_{i,K}^{k-1}, \nonumber \\
	& F_2^m = -\sum_{j=1}^n\sum_{k=1}^{N_T}\Delta t_m\sum_{K\in\T}
	\sum_{\sigma\in\Eint}\tau_\sigma 
	A_{ij}(u^k_\sigma) \text{D}_{K,\sigma} u^k_j \psi_{i,K}^{k-1}, \nonumber \\
	& F_3^m = -\sum_{k=1}^{N_T}\Delta t_m\sum_{K\in\T}\m(K)f_i(u^k_K)\psi_{i,K}^{k-1}.
	\nonumber 
\end{align}
The aim is to show that $F_{j0}^m-F_j^m\to 0$ as $m\to\infty$ for $j=1,2,3$.
Then \eqref{4.FFF} shows that $F_{10}^m+F_{20}^m+F_{30}^m\to 0$, which finishes
the proof.

It is proved in \cite[Theorem 5.2]{CLP03}, using the $L^1(\Omega_T)$ bound
for $u_m$ and the regularity of $\phi$, that $F_{10}^m-F_1^m\to 0$.
Furthermore,
\begin{align*}
  |F_{30}^m-F_3^m| &\le \bigg|\sum_{k=1}^{N_T}\sum_{K\in\T}f_i(u^k_K)\int_K 
	\int_{t_{k-1}}^{t_k} \big(\psi_{i,K}^{k-1} - \psi_i(x,t) \big)dx dt \bigg| \\
	&\le \eta_m\|\psi_i\|_{C^1(\overline{\Omega}_T)} \bigg(\sum_{k=1}^{N_T} 
	\Delta t_m \sum_{K\in\T} \m(K) |f_i(u^k_K)| \bigg).
\end{align*}
We deduce from the growth condition for $f_i$ in Hypothesis (H6) and 
Lemma \ref{lem.est1} that
$$
  |F_{30}^m-F_3^m| \le C\eta_m\|\psi_i\|_{C^1(\overline{\Omega}_T)}
	\bigg(T\m(\Omega) +\sum_{j=1}^n \sum_{k=1}^{N_T}\Delta t_m \|u_j^k\|_{0,2,\T_m}^2
	\bigg) \le C\eta_m\to 0.
$$
The proof of $F_{20}^m-F_2^m\to 0$ is more involved. 
First, we apply discrete integration
by parts and split $F_2^m=F_{21}^m+F_{22}^m$ into two parts with
\begin{align*}
  F_{21}^m &= \sum_{j=1}^n \sum_{k=1}^{N_T}\Delta t_m\sum_{K\in\T}\sum_{\sigma\in\Eint} 
	\tau_\sigma A_{ij}(u^k_K)\text{D}_{K,\sigma} u^k_{j}\text{D}_{K,\sigma} 
	\psi_{i}^{k-1}, \\
  F_{22}^m &=\sum_{j=1}^n\sum_{k=1}^{N_T}\Delta t_m\sum_{K\in\T}
	\sum_{\sigma\in\Eint}\tau_\sigma \big(A_{ij}(u^k_\sigma)-A_{ij}(u^k_K)\big) 
	\text{D}_{K,\sigma} u_j^k\text{D}_{K,\sigma} \psi_{i}^{k-1}.
\end{align*}
The definition of the discrete gradient $\na^m$ in Section \ref{ssec.main} gives
\begin{align*}
  |F_{20}^m - F_{21}^m| 
	&\le \sum_{j=1}^n \sum_{k=1}^{N_T} \sum_{K\in\T} \sum_{\sigma\in\Eint} \m(\sigma) 
	|A_{ij}(u^k_K)| |\text{D}_{K,\sigma} u^k_j| \\ 
	&\phantom{xx}{}\times \bigg|\int_{t_{k-1}}^{t_k}\bigg(\frac{\text{D}_{K,\sigma} 
	\psi_i^{k-1}}{\dist_{\sigma}} - \frac{1}{\m(T_{K,\sigma})} \int_{T_{K,\sigma}} 
	\na\psi_i \cdot \nu_{K,\sigma} dx \bigg) dt \bigg|.
\end{align*}
It is shown in the proof of \cite[Theorem 5.1]{CLP03} that there exists a constant 
$C_0>0$ such that
$$
  \bigg|\int_{t_{k-1}}^{t_k}\bigg(\frac{\text{D}_{K,\sigma} 
	\psi_i^{k-1}}{\dist_{\sigma}} - \frac{1}{\m(T_{K,\sigma})} \int_{T_{K,\sigma}} 
	\na\psi_i \cdot \nu_{K,\sigma} dx \bigg) dt \bigg|
	\le C_0\Delta t_m\eta_m.
$$
Hence, by the Cauchy--Schwarz inequality,
\begin{align*}
  |F_{20}^m-F_{21}^m| &\le C_0\eta_m \sum_{j=1}^n \sum_{k=1}^{N_T} \Delta t_m 
	\sum_{K\in\T} \sum_{\sigma\in\Eint}\m(\sigma)|A_{ij}(u^k_K)| \, 
	|\text{D}_{K,\sigma} u^k_j | \\
  &\le C_0\eta_m\sum_{j=1}^n \sum_{k=1}^{N_T} \Delta t_m \|u_j^k\|_{1,2,\T_m} 
	\bigg(\sum_{K\in\T} |A_{ij}(u^k_K)|^2 \sum_{\sigma\in\Eint}\m(\sigma)\dist_\sigma 
	\bigg)^{1/2}.
\end{align*}
It follows from the mesh regularity \eqref{2.dd} and \cite[(1.10)]{EGH08} that
$$
  \sum_{\sigma\in\Eint}\m(\sigma)\dist_\sigma
	\le \zeta^{-1}\sum_{\sigma\in\Eint}\m(\sigma)\dist(x_K,\sigma)
	\le 2\zeta^{-1}\m(K).
$$
Therefore, applying the Cauchy--Schwarz inequality again, we obtain
\begin{align*}
  |F_{20}^m-F_{21}^m| 
	&\le C(\zeta)\eta_m\sum_{j=1}^n \sum_{k=1}^{N_T} \Delta t_m 
	\|u_j^k\|_{1,2,\T_m}\bigg(\sum_{K\in\T}\m(K)|A_{ij}(u^k_K)|^2\bigg)^{1/2} \\
  &\le C(\zeta)\eta_m\bigg(\sum_{j=1}^n\sum_{k=1}^{N_T} \Delta t_m 
	\|u_j^k\|_{1,2,\T_m}^2\bigg)^{1/2}\bigg(\sum_{j=1}^n\sum_{k=1}^{N_T} \Delta t_m 
	\|A_{ij}(u^k)\|_{0,2,\T_m}^2\bigg)^{1/2}.
\end{align*}
Since $A_{ij}(u^k)$ grows at most linearly,
$$
  |F_{20}^m-F_{21}^m| \le C(\zeta)\eta_m\bigg(\sum_{j=1}^n
	\sum_{k=1}^{N_T} \Delta t_m \|u_j^k\|_{1,2,\T_m}^2\bigg)^{1/2}
	\bigg(\sum_{j=1}^n\sum_{k=1}^{N_T} \Delta t_m 
	\big(1+\|u_j^k\|_{0,2,\T_m}^2\big)\bigg)^{1/2}.
$$
(Here, we see that the $L^{8/3}(\Omega_T)$ estimate of $u_i^k$ for three-dimensional
domains is sufficient.)
The uniform estimates in Lemma \ref{lem.est1} then imply that
$|F_{20}^m-F_2^m|\le C(\zeta)\eta_m\to 0$ as $m\to\infty$.

Finally, we estimate $F_{22}^m$ according to
\begin{align}
  & |F_{22}^m|\le C\eta_m\|\psi_i\|_{C^1(\overline{\Omega}_T)}G^m, \quad\mbox{where} 
	\label{4.F22} \\
  & G^m = \sum_{j=1}^n \sum_{k=1}^{N_T}\Delta t_m \sum_{K\in\T} 
	\sum_{\sigma\in\Eint}\tau_\sigma|A_{ij}(u^k_\sigma)-A_{ij}(u^k_K)|\,
	|\text{D}_{K,\sigma} u_j^k|. \nonumber
\end{align}
Since $A_{ij}$ is assumed to be Lipschitz continuous in Hypothesis (H5) and
$u_{m,i,\sigma}^k\le u_{m,i,K}^k+u_{m,i,L}^k$ for $\sigma\in\E_{\rm int}$
(see \eqref{2.est.usigma}), we deduce from the Cauchy--Schwarz
inequality that
\begin{align*}
  G^m &\le C\sum_{j,\ell=1}^n \sum_{k=1}^{N_T}\Delta t_m \sum_{K\in\T} 
	\sum_{\sigma\in\Eint}\tau_\sigma|u_{\ell,\sigma}^k-u_{\ell,K}^k|\,
	|\text{D}_{K,\sigma}u_j^k| \\
  &\le C\bigg(\sum_{\ell=1}^n\sum_{k=1}^{N_T}\Delta t_m\sum_{\sigma\in\E}\tau_\sigma
	(\text{D}_\sigma u_\ell^k)^2\bigg)^{1/2}\bigg(\sum_{j=1}^n\sum_{k=1}^{N_T}
	\Delta t_m\sum_{\sigma\in\E}\tau_\sigma(\text{D}_\sigma u_j^k)^2\bigg)^{1/2}.
\end{align*}
By Lemma \ref{lem.est1}, the right-hand side is bounded uniformly in $m$.
Thus, we infer from \eqref{4.F22} that $|F_{22}^m|\le C\eta_m\to 0$ and
eventually, $|F_{20}^m-F_2^m| \le |F_{20}^m-F_{21}^m|+|F_{22}^m|\to 0$.
This finishes the proof.


\section{Proof of Theorem \ref{thm.time}}\label{sec.time}

We see from scheme \eqref{2.fvm} after summation over $K\in\T$ that
$u_i^k$ and $u_i^0$ have the same mass and hence, $\|u_i^k\|_{0,1,\T}=\bar u_i$
for $i=1,\ldots,n$.
Summing the entropy inequality \eqref{2.ei} with $C_f=0$ over $k\ge 1$ gives
$$
  c_A\Delta t\sum_{k=1}^\infty\sum_{i=1}^n|u_i^k|_{1,2,\T}^2
  = c_A\Delta t\sum_{k=1}^\infty\sum_{i=1}^n\sum_{\sigma\in\E}\tau_\sigma
  (\textrm{D}_\sigma u_i^k)^2 \le H[u^0].
$$
This shows that the sequence $k\mapsto \sum_{i=1}^n|u_i^k|_{1,2,\T}^2$
converges to zero as $k\to\infty$. The first statement of the theorem then
follows from the discrete Poincar\'e--Wirtinger inequality \cite[Theorem 3.6]{BCF15},
$$
  \|u_i^k-\bar u_i\|_{0,2,\T}\le C_1\zeta^{-1/2}|u_i^k|_{1,2,\T}\to 0
  \quad k\to\infty.
$$

For the second statement, we deduce from the modified entropy inequality \eqref{2.ei2} that
$$
  H[u^k|\bar u] + 4c'_A\Delta t\sum_{i=1}^n\sum_{\sigma\in\E}\tau_\sigma
  (\textrm{D}_\sigma(u_i^k)^{1/2})^2 \le H[u^{k-1}|\bar u].
$$
By the discrete logarithmic Sobolev inequality \cite[Prop.~5.3]{CCHK20},
$$
  H[u^k|\bar u] \le C_2(u^0)\zeta^{-2}\sum_{i=1}^n\sum_{\sigma\in\E}\tau_\sigma
  (\textrm{D}_\sigma(u_i^k)^{1/2})^2,
$$
we find that
$$
  \bigg(1 + \frac{4c'_A\zeta^2}{C_2(u^0)}\Delta t\bigg)H[u^k|\bar u] 
  \le H[u^{k-1}|\bar u].
$$
Setting $\lambda=4c'_A\zeta^2/C_2(u^0)$ and solving the recursion yields
$$
  H[u^k|\bar u]\le (1+\lambda\Delta t)^{-k}H[u^0|\bar u] 
  \le (1+\lambda t_k/k)^{-k}H[u^0|\bar u]
  \le e^{-\lambda t_k}H[u^0|\bar u].
$$
Finally, we apply the discrete Csisz\'ar--Kullback--Pinsker
inequality
$$
  \sum_{i=1}^n \pi_i\|u_i^k-\bar u_i\|_{0,1,\T}^2 \le C_3H[u^k|\bar u],
$$
where $C_3=2\max_{i=1,\ldots,n}\bar u_i$. The proof of this inequality
follows exactly the proof of \cite[Theorem A.3]{Jue16} (just replace integration
over $\Omega$ by summation over $K$). This finishes the proof.


\section{Numerical results}\label{sec.numer}

We present in this section some numerical experiments for the SKT model 
\eqref{1.eq}--\eqref{1.bic} in one and two space dimensions and for two and three 
species. For the two-species SKT model, some of our test cases are inspired by 
\cite{GGJ01,GLS09}. 

\subsection{Implementation of the scheme}

The finite-volume scheme \eqref{2.init}--\eqref{2.usigma} is implemented in MATLAB. 
Since the numerical scheme is implicit in time, we have to solve a nonlinear system 
of equations at each time step.
In the one-dimensional case, we use Newton's method. Starting from 
$u^{k-1}=(u^{k-1}_1, u^{k-1}_2)$, we apply a Newton method with precision 
$\eps = 10^{-10}$ to approximate the solution to the scheme at time step $k$.
In the two-dimensional case, we use a Newton method complemented by an adaptive 
time-stepping strategy to approximate the solution of the scheme at time $t_k$. 
More precisely, starting again from $u^{k-1}=(u^{k-1}_1, u^{k-1}_2)$, we launch a 
Newton method. If the method does not converge with precision 
$\eps= 10^{-8}$ after at most $50$ steps, we halve the time step size and restart the 
Newton method. At the beginning of each time step, we double the previous time 
step size. Moreover, we impose the condition 
$10^{-8}\leq \Delta t_{k-1} \leq 10^{-2}$ with 
an initial time step size $\Delta t_0 = 10^{-5}$.

\subsection{Test case 1: Rate of convergence in space}

In this section, we illustrate the order of convergence in space for the two-species 
SKT model in one space dimension with $\Omega = (-\pi,\pi)$. We choose the
coefficients $a_{i0}=0.05$ and $a_{ii} =2.5\cdot 10^{-5}$ for $i=1,2$, 
$a_{12}=1.025$ and $a_{21}=0.075$. We take rather stiff values of the Lotka--Volterra constants as in \cite[Section 3.3]{GLS09},
$b_{10}=59.7$, $b_{20}=49.75$, $b_{11}=24.875$, and 
$b_{12}=b_{21}=b_{22}=19.9$. Finally, we impose the initial datum
\begin{align*}
  & u_1^0(x,y) = 2 + 0.31 f(x-0.25) + 0.31 f(x-0.75),\quad u_2^0(x,y) = 0.5, \\
	& \mbox{where }f(x)=\max\{1 - 8^2 x^2, 0\}. 
\end{align*}
Since exact solutions to the SKT model are not explicitly known, we compute 
a reference solution on a uniform mesh composed of $5120$ cells and with 
$\Delta t = (1/5120)^2$. We use this rather small value of $\Delta t$
because the Euler discretization in time exhibits a first-order convergence rate, 
while we expect, as observed for instance in \cite{CG20}, a second-order convergence 
rate in space for scheme \eqref{2.init}-\eqref{2.usigma}, due to the logarithmic 
mean used to approximate the mobility coefficients in the numerical fluxes. 
We compute approximate solutions on uniform meshes made of $40$, $80$, $160$, 
$320$, $640$, and $1280$ cells, respectively. In Table \ref{table.conv}, we present
the $L^2(\Omega)$ norm of the difference between the approximate solutions and 
the average of the reference solution $u_{\rm ref}$ at the final time $T=10^{-3}$. 
As expected, we observe a second-order convergence rate in space.

\begin{table}[!ht]
\begin{center}
\begin{tabular}{|C{1cm}|C{2cm}|C{1cm}|C{2cm}|C{1cm}|}
  \hline
   \multirow{2}*{cells} 
    & \multicolumn{2}{c|}{$u_1$} & \multicolumn{2}{c|}{$u_2$}\\
   \cline{2-5}   
   & $L^2$ error & order & $L^2$ error & order \\
  \hline
  40 & 8.2518e-04 &   & 2.6979e-05 &  \\
  80 & 2.1542e-04 & 1.94 & 1.2174e-05 & 1.15  \\
  160 & 5.5456e-05 & 1.96 & 4.2493e-06 & 1.52 \\
  320 & 1.3889e-05 & 2.00 & 1.0963e-06 & 1.95  \\
  640 & 3.4352e-06 & 2.02 & 2.7278e-07 & 2.01  \\
  1280 & 8.1811e-07 & 2.07 & 6.5056e-08 & 2.07 \\
  \hline
\end{tabular}
\caption{$L^2(\Omega)$ norm of the difference $u_i-u_{{\rm ref},i}$
in space at final time $T=10^{-3}$.}
\label{table.conv}
\end{center}
\end{table}


\subsection{Test case 2: Pattern formation}\label{ssec.pattern}

We illustrate the formation of spatial pattern exhibited by the two-species SKT model
in the two-dimensional domain $\Omega = (0,1)^2$ with a mesh composed of $3584$ 
triangles. The diffusion and Lotka--Volterra coefficients are chosen as in
test case 1.
For these values, the stable equilibrium for the Lotka--Volterra ODE system
is given by $u^*=(2,0.5)$ (see, e.g., \cite{TLP10}).
The initial datum is a perturbation of the constant equilibrium:
\begin{align}
  & u_1^0(x,y) = 2 + 0.31 g(x-0.25,y-0.25) + 0.3 g(x-0.75, y-0.75),\quad
  u_2^0(x,y) = 0.5, \nonumber \\
  & \mbox{where }g(x,y) = \max\{1 - 8^2 x^2 - 8^2 y^2,0\}. \label{8.def.g}
\end{align}
In Figure \ref{Fig.pattern}, we show the evolution of the densities $u_1$ and $u_2$ 
at different times. At time $t=0.5$, the solution $(u_1,u_2)$ seems to converge 
towards the constant equilibrium state $u^*$. However, due to the cross-diffusion 
terms, we observe after this transient time the formation of spatial patterns, 
which indicate that the state $(2,0.5)$ is unstable for the PDE system.

Indeed, it is proved in \cite[Theorem 3.1]{TLP10} that the constant linearly stable
equilibrium $u^*$ for the Lotka--Volterra system is unstable for the SKT model if certain conditions are satisfied. To this end, we introduce the matrices
$$
  D^* = \begin{pmatrix}
  a_{10} + 2a_{11}u_1^* + a_{12}u_2^* & a_{12}u_1^* \\
  a_{21}u_2^* & a_{20} + a_{21}u_1^* + 2a_{22}u_2^*
  \end{pmatrix}, \quad J^*=\na_u f(u^*).
$$
The conditions are as follows: 
(i) $\operatorname{trace}(D^*)>0$, (ii) $\det(D^*)>0$,
(iii) $\det(D^*)+\det(J^*)>0$,
(iv) there exists at least one positive eigenvalue $\mu$ of the Neumann
problem $-\Delta v=\mu v$ in $\Omega$, $\na v\cdot\nu=0$ on $\pa\Omega$ such that
$0<k_-\le\mu\le k_+$, where $k_\pm$ are the solutions to the quadratic equation
$$
  \det(D^*) k^2 + (\det(J^*)+\det(D^*))k + \det(J^*) = 0.
$$
With our chosen values, we have
$\operatorname{trace}(J^*)\approx -59.7<0$, $\det(J^*)\approx 99.0025>0$
(this implies that $u^*$ is stable for the Lotka--Volterra system) and
$\operatorname{trace}(D^*)\approx 0.7626 > 0$, $\det(D^*)\approx 0.00357>0$.
The eigenvalues of the Neumann problem on $(0,1)^2$ are given by
$\mu_p=(p_1\pi)^2 + (p_2^2\pi)^2$ for $p=(p_1,p_2)\in{\mathbb N}^2$
\cite[Section 3.1]{GrNg13}.
A computation shows that $k_+\approx 129.82$ and $k_-\approx 1$.
The assumptions of \cite[Theorem 3.1]{TLP10} are satisfied and therefore,
$u^*$ is an unstable equilibrium for the SKT model. Moreover, because of
$$
  \frac{b_{10}}{b_{11}} < \frac{b_{20}}{b_{21}} \quad \mbox{and} \quad 
	\frac{b_{20}}{b_{22}} < \frac{b_{10}}{b_{12}},
$$
the two species coexist \cite[Section 6.2]{SK97}. These theoretical results
confirm our numerical outcome. 

\begin{figure}
\begin{center}
\includegraphics[scale=0.35]{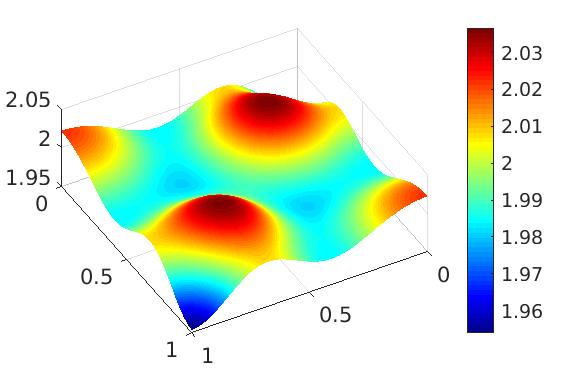}
\includegraphics[scale=0.35]{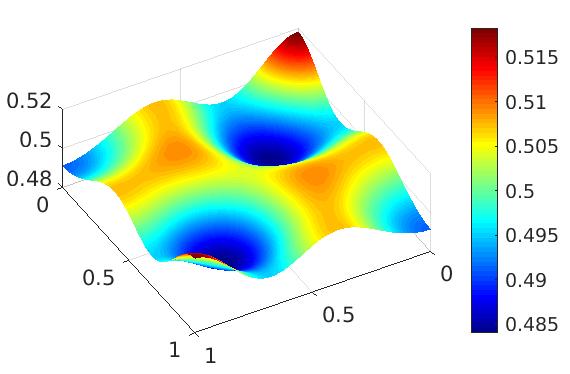}
\includegraphics[scale=0.35]{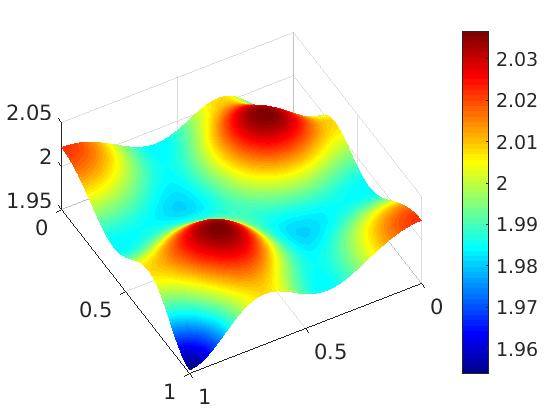}
\includegraphics[scale=0.35]{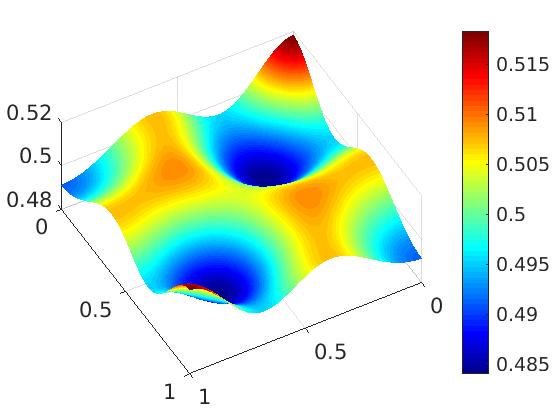}
\includegraphics[scale=0.35]{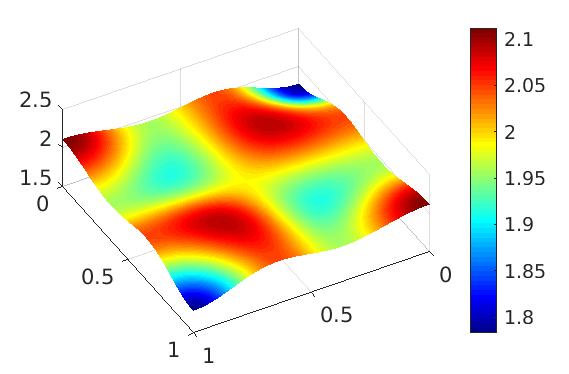}
\includegraphics[scale=0.35]{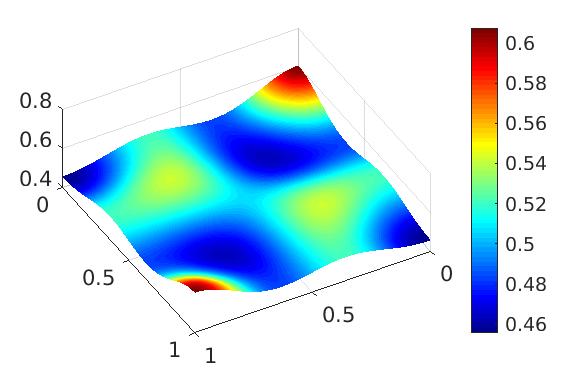}
\includegraphics[scale=0.35]{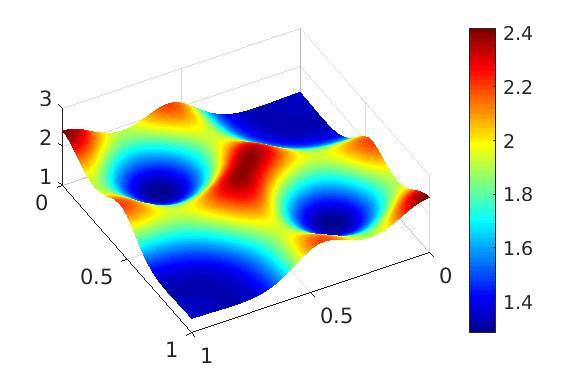}
\includegraphics[scale=0.35]{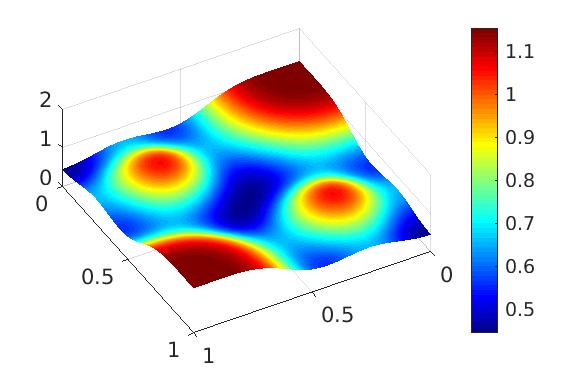}
\end{center}
\caption{Test case 2: Evolution of the densities $u_1$ (left column) 
and $u_2$ (right column) at times $t=0.5$, $1$, $2$ and $4$ (from top to bottom).}
\label{Fig.pattern}
\end{figure}

\subsection{Test case 3: Spatial niche and repulsive potential}\label{ssec.niche}

In this section, we consider the two-species SKT model with environmental potential, 
i.e., we add to equation \eqref{1.eq} a smooth function $\phi(x)$,
$$
  \pa_t u_i - \diver\bigg(\sum_{j=1}^n A_{ij}(u)\na u_j-d_i u_i \na \phi \bigg) 
  = f_i(u)\quad\mbox{in }\Omega,\ t>0,\ i=1,\ldots,n,
$$
where $d_i > 0$ and $A_{ij}(u)$ is given by \eqref{1.A} for $i$, $j=1,\ldots,n$. 
We adapt the definition of the finite-volume scheme \eqref{2.init}--\eqref{2.usigma} 
by defining the fluxes as
$$
  \mathcal{F}_{i,K,\sigma}^k = - \bigg(\sum_{j=1}^n \tau_\sigma A_{ij}(u^k_\sigma) 
  \textrm{D}_{K,\sigma} u^k_j - d_i u^k_{i,\sigma} \textrm{D}_{K,\sigma}\phi\bigg) 
  \quad\mbox{for }K\in \T,\ \sigma \in \E_K,
$$
where
$$
  \phi_K = \frac{1}{\m(K)} \int_K \phi(x) \, dx \quad\mbox{for }K \in \T.
$$
By adapting the proof of Theorem \ref{thm.ex}, we obtain the following discrete 
entropy inequality:
\begin{align*}
  (1-&C_f\Delta t)H[u^k] + \frac{c_A}{2}\Delta t\sum_{i=1}^n\sum_{\sigma\in\E}
  \tau_\sigma(\mathrm{D}_\sigma u_i^k)^2 \\ 
  &\le H[u^{k-1}] + \frac{\Delta t}{2c_A} \sum_{i=1}^n d_i^2 \sum_{\sigma \in \E}
  \tau_\sigma (\mathrm{D}_\sigma \phi)^2 + C_f\Delta t\m(\Omega), \quad k\ge 1.
\end{align*}
This estimate ensures the existence of a nonnegative solution to the scheme and 
its convergence to the continuous model.

Now we consider a mesh of $\Omega = (0,1)^2$ composed of $3584$ triangles and 
choose the same values for the diffusion and Lotka--Volterra constants as
in Section \ref{ssec.pattern}. Furthermore, we take $d_1 = d_2 = 2$ and the 
environmental potential
$$
  \phi(x,y) = \exp\big(-2 \big((x-0.5)^2 + (y-0.5)^2\big)\big).
$$
The inital data is defined according to
\begin{align*}
  u_1^0(x,y) &= 2 + 0.31g(x-0.25,y-0.25) + 0.31g(x-0.75, y-0.75), \\
  u_2^0(x,y) &= 0.5 + 0.2 g(x-0.5,y-0.5),
\end{align*}
where the function $g$ is given by \eqref{8.def.g}.

In Figure \ref{Fig.niches} we illustrate the creation of an ecological niche. We observe that species 2 creates a niche around the point $(0.5,0.5)$ to avoid extinction even when dominated by species 1. This figure can be seen as a 
two-dimensional variant of the numerical experiments done in 
\cite[Section 3.2, case I]{GLS09}.

\begin{figure}[!ht]
\begin{center}
\includegraphics[scale=0.35]{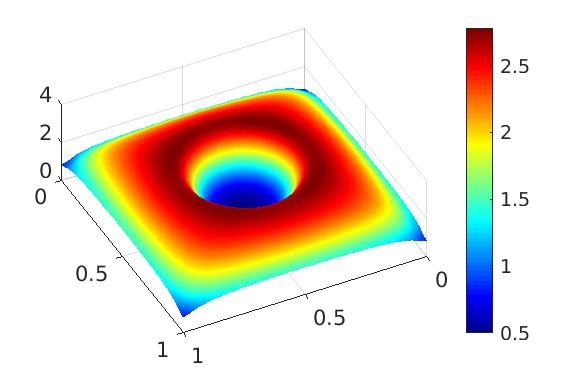}
\includegraphics[scale=0.35]{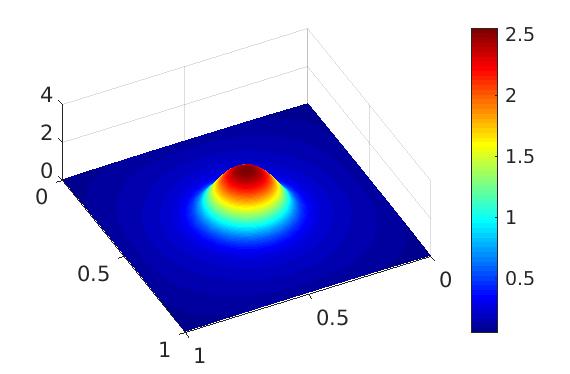}
\end{center}
\caption{Test case 3: Density $u_1$ (left) and $u_2$ (right) at time $t=0.5$.}
\label{Fig.niches}
\end{figure}

\subsection{Test case 4: Convergence to a constant steady state}\label{ssec.time}

In this last numerical experiment, we illustrate Theorem \ref{thm.time}. In particular, we consider the SKT model with three species and without source terms. 
We choose the values $a_{10}=1$, $a_{20}=5$, $a_{30} = 7$, $a_{11}=a_{21}=a_{31}=1$, 
$a_{12}=3$, $a_{22}=2$, $a_{32}=1$, $a_{13}=4$, $a_{23}=4/3$, $a_{33}=2$, 
with $\pi_1=1$, $\pi_2=3$ and $\pi_3=4$ and the initial datum
\begin{align*}
  u_1^0(x,y)&= 0.5 \mathbf{1}_{(0.2,0.4)^2}(x,y),\\ 
  u_2^0(x,y)&= 0.7 \mathbf{1}_{(0.6,0.8)\times(0.2,0.4)}(x,y),\\
  u_3^0(x,y)&= \mathbf{1}_{(0.4,0.6) \times (0.6,0.8)}(x,y).
\end{align*}
In Figure \ref{Fig.longtime}, we present in semilogarithmic scale the behavior 
of the relative Boltzmann entropy
$$
  H[u^k|\bar{u}] = \sum_{i=1}^n \sum_{K \in \T} \m(K) \pi_i \bigg( u^k_{i,K} 
  \log\bigg(\frac{u^k_{i,K}}{\bar{u}_i}\bigg) + \bar{u}_i - u^k_{i,K} \bigg), 
$$
where $\bar{u}_i = \m(\Omega)^{-1} \int_\Omega u^0_i(x) \, dx$ for $i=1,\ldots,n$, 
and the squared weighted $L^1$ norm
$$
  \sum_{i=1}^n \pi_i \|u^k_i - \bar{u}_i \|^2_{0,1,\T},
$$
versus time (with final time $T=1$) for a mesh of $\Omega = (0,1)^2$ composed of $3584$ triangles. As proved in Theorem \ref{thm.time}, we observe an exponential 
convergence rate of the solutions to the scheme towards the constant steady state.

\begin{figure}[!ht]
\begin{center}
\includegraphics[scale=1]{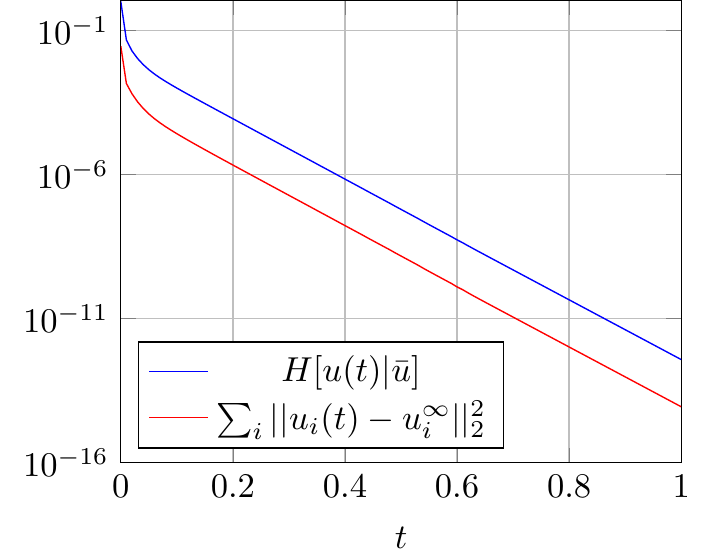}
\end{center}
\caption{Test case 4: Evolution of the relative Boltzmann entropy and the squared weighted $L^1$ norm in semilogarithmic scale.}
\label{Fig.longtime}
\end{figure}


\begin{appendix}
\section{Computation of the constant $C_f$ given by \eqref{skt.defCf}}\label{app}

We claim that the Lotka--Volterra terms \eqref{1.f} satisfy
Hypothesis (H6) with $C_f$ given by \eqref{skt.defCf}.
For this, define $g(s)=s(\log s-1)+1$ for $s\ge 0$.
Then, for $u=(u_1,\ldots,u_n) \in (0, \infty)^n$, using $-u_i \log u_i \leq e^{-1}$, 
we have
\begin{align*}
  \sum_{i=1}^n f_i(u)\pi_i\log(u_i) 
	&\le \sum_{i=1}^n \pi_i b_{i0} g(u_i) + \sum_{i=1}^n \pi_i b_{i0}u_i 
	- \sum_{i=1}^n \pi_i u_i \log(u_i) \sum_{j=1}^n b_{ij} u_j \\
  &\le h(u) \max_{i=1,\ldots,n} b_{i0} + \sum_{i=1}^n \pi_i b_{i0} u_i 
	+ \frac{1}{e} \sum_{i=1}^n \pi_i \sum_{j=1}^n b_{ij} u_j.
\end{align*}
Reordering the terms in the last sum yields
\begin{align*}
  \sum_{i=1}^n f_i(u) \pi_i \log(u_i) 
	&\le h(u)\max_{i=1,\ldots,n} b_{i0} + \sum_{i=1}^n \pi_i u_i 
	\bigg(b_{i0} + \frac{1}{e\pi_i} \sum_{j=1}^n \pi_j b_{ji} \bigg) \\
  &\le h(u) \max_{i=1,\ldots,n} b_{i0} + \max_{i=1,\ldots,n} 
	\bigg(b_{i0} + \frac{1}{e\pi_i} \sum_{j=1}^n \pi_j b_{ji} \bigg) 
	\sum_{i=1}^n \pi_i u_i.
\end{align*}
Now, we simply notice that the inequality $s\le (1+g(s))/\log 2$ holds true
for all $s\ge 0$. Then, for $s=\sum_{i=1}^n \pi_i u_i$, we infer that
\begin{align*}
  \sum_{i=1}^n f_i(u) \pi_i \log(u_i) 
	&\le (1+h(u))\max_{i=1,\ldots,n} b_{i0} 
  + \frac{1+h(u)}{\log 2}\max_{i=1,\ldots,n} 
  \bigg(b_{i0} + \frac{1}{e\pi_i} \sum_{j=1}^n \pi_j b_{ji} \bigg) \\
	&\le \frac{2}{\log 2}\max_{i=1,\ldots,n} 
  \bigg(b_{i0} + \frac{1}{e\pi_i} \sum_{j=1}^n \pi_j b_{ji} \bigg)(1+h(u)) 
	= C_f(1+h(u)).
\end{align*}
which shows the claim.

\end{appendix}


\end{document}